\documentclass[11pt,reqno]{article}
\usepackage{authblk}
\usepackage[margin=1in] {geometry}
\usepackage{datetime}
\usepackage{amssymb,latexsym,amsmath}
 \usepackage[mathscr]{eucal}
\usepackage{amsmath,amsxtra,amssymb,latexsym,amscd,amsthm,amsfonts,amstext,makeidx,upgreek,dsfont,xcolor}
\usepackage{indentfirst}
\usepackage{enumerate}
\usepackage{comment}
\usepackage{multicol,multirow}
\allowdisplaybreaks

\def\Dbar{\leavevmode\lower.6ex\hbox to 0pt{\hskip-.03ex\accent"16\hss}D}

\usepackage{setspace}
\usepackage[colorlinks=true,linkcolor=blue,citecolor=magenta]{hyperref} 
\usepackage{graphicx}
\usepackage{makeidx,esint}

\providecommand{\keywords}[1]{\textbf{\textit{Key words---}} #1} 

\usepackage{multirow}
\usepackage{hhline}
\usepackage{epstopdf}

\newtheorem{theorem}{{\bf Theorem}}[section]

\theoremstyle{plain} 
\newtheorem{lemma}[theorem]{Lemma} 
\newtheorem{proposition}[theorem]{Proposition}

\newtheorem{remark}[theorem]{Remark}
\newtheorem{definition}[theorem]{Definition}

\usepackage{multicol,amstext,amsmath,amsthm,latexsym,amsbsy,amsfonts,amssymb,mathrsfs}
\usepackage{functan,graphicx,verbatim,mathtools} 
\usepackage{gensymb}

\newcommand{\R}{\mathbb R}
\renewcommand{\epsilon}{\varepsilon}
\newcommand{\eps}{\epsilon}
\newcommand{\ue}{u^{\eps}}
\newcommand{\de}{d^{\eps}}
\newcommand{\MR}{\mathsf{MR}}
\newcommand{\bra}[1]{\left(#1\right)}
\newcommand{\sbra}[1]{\left[#1\right]}
\newcommand{\na}{\nabla}
\newcommand{\LO}[1]{L^{#1}(\Omega)}
\newcommand{\LQ}[1]{L^{#1}(Q_T)}
\newcommand{\pa}{\partial}
\newcommand{\intQT}{\iint_{Q_T}}
\renewcommand{\H}{\mathbf{H}}

\newcommand{\wt}[1]{\widetilde{#1}}
\newcommand{\M}{\mathcal M}

\usepackage[new]{old-arrows}

\title{Rigorous derivation of Michaelis-Menten kinetics \\ in the presence of diffusion}

\author[$\dagger$]{Bao Quoc Tang}
\author[$\dagger$,$\ddagger$]{Bao-Ngoc Tran}
\affil[$\dagger$]{\small{\textit{Institute of Mathematics and Scientific Computing, University of Graz, 
	Heinrichstrasse 36, 8010 Graz, Austria}}}
\affil[$\ddagger$]{\small{\textit{Department of Mathematics, Faculty of Science, Nong Lam University, Ho Chi Minh City, Vietnam}}}
\affil[ ]{\small{\textit{quoc.tang@uni-graz.at, bao-ngoc.tran@uni-graz.at\thanks{Corresponding author.}}}}
\date{}

\begin{document}

\maketitle

\begin{abstract} 
	Reactions with enzymes are critical in biochemistry, where the enzymes act as catalysis in the process. One of the most used mechanisms for modeling enzyme-catalyzed reactions is the Michaelis-Menten (MM) kinetic. In the ODE level, i.e. concentrations are only on time-dependent, this kinetic can be rigorously derived from mass action law using quasi-steady-state approximation. This issue in the PDE setting, for instance when molecular diffusion is taken into account, is considerably more challenging and only formal derivations have been established. In this paper, we prove this derivation rigorously and obtain MM kinetic in the presence of spatial diffusion. In particular, we show that, in general, the reduced problem is a cross-diffusion-reaction system. Our proof is based on improved duality method, heat regularisation and a suitable modified energy function. To the best of our knowledge, this work provides the first rigorous derivation of MM kinetic from mass action kinetic in the PDE setting.
\end{abstract}  
\keywords{Enzyme reactions; Michaelis-Menten kinetics; Improved duality method; Modified energy method; Cross-diffusion}

\tableofcontents

\section{Introduction} 
\subsection{Problem setting}
Enzyme-catalyzed reactions are critical in biochemistry, where the reaction rates can be   accelerated by well over a million-fold comparing with non catalyzed reactions, 
and  enzymes are eventually not getting either consumed or   transformed  into another substance. These reactions are typified by combination of an enzyme $E$ with its substrate molecule $S$ to form an enzyme-substrate complex $C$, which is a needed step in enzyme mechanism and the key to study kinetic behaviors, see Henri \cite{henri1903lois,henri2006theorie}. Diagrammatically, enzyme reaction mechanism reads    
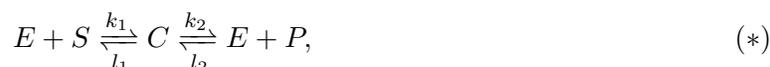
\begin{align}
E + S \xrightleftharpoons[l_1]{k_1}   C  \xrightleftharpoons[l_2]{k_2} E + P, \tag{\textasteriskcentered} \label{star}
\end{align}
where the substrate $S$ binds the enzyme-catalyst  $E$ to form the complex $C$ that can be transformed into the enzyme and a product $P$. The numbers   $k_1,l_1,k_2\in(0,\infty)$ and $l_2\in [0,\infty)$  are reaction rate constants. If $l_2=0$, then (\ref{star}) is called irreversible, i.e., the enzyme and product cannot react to synthesize the complex. In reversible case,    $l_2>0$,  there is a  backward reaction that forms   the complex from the enzyme and product. Assume for a moment that the reaction is irreversible, i.e. $l_2 = 0$, and the concentrations of enzyme ($e$), substrate ($s$), complex ($c$), and product ($p$) depend only on time. By applying the mass action law, one gets the differential system with mass action kinetic
\begin{equation}\label{ODEsys}
\begin{cases}
	s' = -k_1se + l_1c,\\
	e' = -k_1se + (l_1+k_2)c,\\
	c' = k_1se - (l_1+k_2)c,\\
	p' = k_2c
\end{cases}
\end{equation}
with initial data $z(0) = z_0, z\in \{s,e,c,p\}$. From the conservation laws $(e+c)' = 0$ and $(s+c+p)' = 0$, it leads to the reduced system
\begin{equation*}
	\begin{cases}
		s' = -k_1s(e_0+c_0 - c) + l_1 c\\
		c' = k_1s(e_0 + c_0 - c) - (l_1+k_2)c.
	\end{cases}
\end{equation*}
  The quasi-steady-state-approximation hyphothesis assumes that the complex concentration $c$ reaches its equilibrium almost instantly, that is we have $0 \approx c' =  k_1s(e_0+c_0-c) - (l_1+k_2)c$, which leads to the evolution equation of the substrate with the famous \textit{Michaelis-Menten kinetic} (or briefly, MM kinetic)
\begin{equation}\label{ee}
	s' = -\frac{k_1k_2(e_0+c_0)s}{k_1s + l_1 + k_2}.
\end{equation}
This type of reaction rate has become one of the most, if not the most, used kinetics in enzymatic, or more generally catalytic  reactions in the literature. The aforementioned derivation of MM kinetic from mass action kinetic can be rigorously justified (see e.g. \cite{perthame2015parabolic,segel1989quasi}) under different conditions, in which the assumption that the initial enzyme concentration is sufficiently small, i.e. $e_0/s_0 \ll 1$, is of particular importance since it is biologically relevant. 

\medskip
In many contexts, e.g. experiments in vivo, chemical concentrations are spatially inhomogeneous as diffusion is hindered in the gel-like cytosol, and the cytosolic composition varies in different regions of the cell. In such scenarios, concentrations of chemicals are functions of both temporal and spatial variables, and it is natural to consider reaction-diffusion systems. Assume that the enzyme-catalysed reaction (\ref{star}) occurs in a bounded vessel $\Omega\subset \mathbb{R}^N$, $N\ge 1$, with smooth boundary $\partial\Omega$,  and the  diffusion process follows the Fick's second law. Let $\delta_S,\delta_E,\delta_C,\delta_P>0$ be the molecular diffusion rates of $S,E,C,P$, 
and $  n_S,n_E,n_C,n_P $ be the concentrations  of $S,E,C,P$, which are  functions of $(x,\tau)\in \Omega \times [0,\infty)$. Thanks to the law of mass action, the governing equations for (\ref{star}) are\footnote{The reaction-diffusion system \eqref{MainEquation0} can be derived from reactive Boltzmann system describing the enzyme reaction \eqref{star}, see \cite{bisi2006reactive}.}  
\begin{align}
\left\{ \begin{array}{rcrcrcrcrcr}
\displaystyle \partial _{\tau} n_S  &-& \delta_S \Delta n_S &=&  - \, k_1 n_Sn_E &+& l_1n_C  &  & & \text{in } Q_T,\\
\displaystyle \partial _{\tau} n_E  &-& \delta_E \Delta n_E &=& - \, k_1 n_S n_E &+&  (l_1+k_2)n_C &-& l_2n_En_P    & \text{in } Q_T,\\
\displaystyle \partial _{\tau} n_C   &-& \delta_C \Delta n_C &=&   + \,  k_1 n_Sn_E &-&  (l_1+k_2)n_C &+& l_2n_En_P    & \text{in } Q_T,\\
\displaystyle \partial _{\tau} n_P  &-&  \delta_P \Delta n_P &=&   &+& k_2n_C &-& l_2n_En_P      & \text{in } Q_T. 
\end{array}
\right. \label{MainEquation0}
\end{align}
where     $Q_T:=\Omega \times (0,T)$ for $T\in (0,\infty]$. Denote by  $\nu=\nu(x)$ the unit outer normal vector at point $x\in \partial\Omega$. This system is subjected to the homogeneous Neumann boundary condition 
\begin{align}
 \pa_{\nu}n_S =  \pa_{\nu}n_E =  \pa_{\nu}n_C =  \pa_{\nu}n_P = 0  \quad \text{on }   \pa\Omega\times(0,T),  \label{BoundaryCondition0}
\end{align}
where $\pa_{\nu}f = \na f \cdot \nu$ is the directional derivative. The initial states are given by  
\begin{align}
\big( n_S( x,0),n_E( x,0),n_C( x,0),n_P( x,0)\big) = \big( n_{SI}(x), n_{EI}(x) , n_{CI}(x), n_{PI}(x)\big), \quad x\in \Omega. \label{InitialCondition0}
\end{align}
Naturally, one would expect that the MM kinetic is also a suitable reaction rate for the substrate concentration, and in fact this has been used extensively in the literature. Unlike the case of the differential system \eqref{ODEsys}, the derivation of MM kinetic from mass action law through the reaction-diffusion system \eqref{MainEquation0} is more challenging. For instance, it is obvious that the conservation laws are not any more pointwise (unless both enzyme $E$ and complex $C$ are not diffusing) and a direct reduction using conservation laws is not possible. One can  use asymptotic analysis to at least formally derive the MM kinetic from \eqref{MainEquation0}. More precisely, we consider the following small parameter
\begin{equation}\label{a1}
	\eps:= \frac{\int_{\Omega}n_{EI}(x)dx}{\int_{\Omega}n_{SI}(x)dx} \ll 1
\end{equation}
which mimics the basic  quasi-steady state or pseudo-steady state hypothesis that the initial  concentration of enzyme is much less than the initial substrate  concentration \cite{segel1989quasi,goeke2015determining}. It is remarked that this is considerably more general than imposing
\begin{align*}
\frac{n_{EI}(x)}{n_{SI}(x)} \ll 1 \quad \text{or} \quad \frac{n_{EI}(x)+n_{CI}(x)}{n_{SI}(x)} \ll 1  \quad \text{ for all } x\in \Omega   
\end{align*}
since \eqref{a1} allows $n_{SI}$ to be zero on some non-zero measured set or $n_{EI}$ to have spikes. 
Following the rescaling in \cite[Section 3]{FraLaxWalWit18},
\begin{align}
\tau =\frac{t}{\epsilon}, \quad d_z  = \frac{\delta_z}{\epsilon},   \quad   \widetilde n_{E}(x,t)=\frac{n_E(x,\tau)}{\epsilon}, \quad   \widetilde n_C(x,t)=\frac{n_C(x,\tau)}{\epsilon},\\
\widetilde{n}_S(x,t) = n_S(x,\tau), \quad \widetilde{n}_P(x,t) = n_P(x,\tau),\quad z\in \{S,E,C,P\} \label{MainScaling}
\end{align}
and denoting  $\ue = (\ue_j)_{1\le j\le 4}$ where
$$\ue_1 =  \widetilde{n}_S, \; \ue_2 = \widetilde  n_E,\; \ue_3 = \widetilde  n_C,\; \ue_4 =  \widetilde{n}_P, \quad \text{ and } \quad  (d_j^{\epsilon})_{1\le j\le 4}=(d_S,d_E,d_C,d_P),$$  
we obtain from (\ref{MainEquation0})-(\ref{InitialCondition0}) the following $\eps$-dependent reaction-diffusion system
\begin{align}
 \left\{ \begin{array}{rllllll}
\partial_t u_1^\epsilon -  d_1^{\epsilon} \Delta u_1^\epsilon &=&  -\, k_1 u_1^\epsilon u_2^\epsilon   +   l_{1} u_3^\epsilon  & \text{in } Q_T, \vspace*{0.15cm}\\
\partial_t u_2^\epsilon  -  d_2^{\epsilon} \Delta u_2^\epsilon &=& - \, \dfrac{1}{\epsilon} \Big( k_1 u_1^\epsilon u_2^\epsilon  -    (  k_{2} + l_{1} ) u_3^\epsilon  +    l_2   u_2^\epsilon u_4^\epsilon  \Big)  & \text{in } Q_T, \vspace*{0.15cm} \\
\partial_t u_3^\epsilon   -  d_3^{\epsilon} \Delta u_3^\epsilon  &=& + \, \dfrac{1}{\epsilon} \Big( k_1 u_1^\epsilon  u_2^\epsilon   -    (  k_{2} + l_{1} ) u_3^\epsilon   +    l_2  u_2^\epsilon u_4^\epsilon  \Big)  & \text{in } Q_T, \vspace*{0.15cm} \\
\partial_t u_4^\epsilon  - d_4^{\epsilon} \Delta u_4^\epsilon  &=&    - \, l_2 u_2^\epsilon  u_4^\epsilon  +  k_{2} u_3^\epsilon   & \text{in } Q_T, \vspace*{0.15cm} \\
 \partial_\nu u^\epsilon   &=&   0   & \text{on } \pa\Omega\times(0,T), \vspace*{0.15cm} \\
u^\epsilon (0) &=&    u_{0}^\epsilon    & \text{in } \Omega ,
\end{array} 
\right.  \label{MainProblem}
\end{align}
where $\partial_\nu u^\epsilon  $ shortly stands for $(\partial_\nu u_j^\epsilon  )_{1\le j\le 4}$.  
Through out this paper, without loss of generality, we assume that the initial data $u_0^\eps$ does not depend on $\epsilon$, and consequently we will remove the superscript. 
Moreover, we consider the case of \textit{slow diffusion} $\delta_z = O(\eps)$, $z\in \{S,E,C,P\}$. This is achieved by assuming that the diffusion rates $d_j^\epsilon$ are convergent, i.e.,  
\begin{align}
\lim\limits_{\epsilon\to0^+}   d_j^\epsilon   =   d_j  \in (0,\infty)^4, \, j=1,\dots,4, \label{AssumpOnDj} 
\end{align}
which we impose  {\it throughout this paper}. From the second and third equations of \eqref{MainProblem}, we expect, at least formally, that when $\eps \to 0$, 
\begin{equation*}
	k_1\ue_1\ue_2 - (k_2+l_1)\ue_3 + l_2\ue_2\ue_4 \to 0 = k_1u_1u_2 - (k_2+l_1)u_3 + l_2u_2u_4,
\end{equation*}
where it is implicitly assumed that $\ue_j \to u_j$, $j\in \{1,\ldots,4\}$. It follows that
\begin{equation}\label{CriticalManifold}
	u_3 = \frac{k_1u_1+l_2u_4}{k_1u_1+l_2u_4+k_2+l_1}(u_2+u_3).
\end{equation}
By summing the second and third equations of \eqref{MainProblem} and let $\eps\to 0$, we formally have
\begin{equation*}
\begin{aligned}
	\pa_t(u_2 + u_3) - d_2\Delta(u_2+u_3) &= (d_3-d_2)\Delta u_3\\
	&= (d_3-d_2)\Delta\left[ \frac{k_1u_1+l_2u_4}{k_1u_1+l_2u_4+k_2+l_1}(u_2+u_3)\right].
\end{aligned}
\end{equation*}
Denote by $v = u_2 + u_3$, the system \eqref{MainProblem} in the limit $\eps\to 0$ is formally reduced to the following \textit{cross-diffusion-reaction system}
\begin{equation}\label{ReducedSys}
\left\{
	\begin{aligned}
		&\pa_t u_1 - d_1\Delta u_1  = -\frac{(k_1k_2u_1 - l_1l_2u_4)v}{k_1u_1+l_2u_4 + k_2 + l_1},\\
		&\pa_t v - d_2\Delta v = (d_3-d_2)\Delta\left(\frac{(k_1u_1+l_2u_4)v}{k_1u_1+l_2u_4+k_2+l_1} \right),\\
		&\pa_t u_4 - d_4\Delta u_4 = +\frac{(k_1k_2u_1 - l_1l_2u_4)v}{k_1u_1 + l_2u_4 + k_2 + l_1},
	\end{aligned}
\right.
\end{equation}
subjected to homogeneous Neumann boundary conditions, and initial data $u_1(0) = u_{10}$, $v(0) = u_{20}+u_{30}$, and $u_4(0) = u_{40}$. Consider the case of irreversible enzyme reaction in \eqref{star}, i.e. $l_2 =0$, and $d_2 = d_3$, we see that $v$ can be solved independently from $\pa_t v - d_2\Delta v = 0$, and therefore can be considered known for the equation of the substrate $u_1$, which now reads
\begin{equation*}
	\pa_t u_1 - d_1\Delta u_1 = \frac{-k_1k_2u_1v}{k_1u_1 + k_2 + l_1} . 
\end{equation*}
This is precisely the well-known \textit{MM kinetic} which has been used frequently in the literature. System \eqref{ReducedSys} indicates that, in general, one should take into account the evolution of $v$ featuring cross-diffusion.  This formal derivation and related versions were given in e.g. \cite{FraLaxWalWit18,kalachev2007reduction}, but up the best of our knowledge, no rigorous proof has been investigated. 
Our present paper shows, therefore, for the first time, that the reduction from \eqref{MainProblem} to \eqref{ReducedSys} is rigorous rather than just a formal derivation. 

\subsection{State of the art}
The rigorous derivation of MM kinetic in the case homogeneous setting, i.e. for the differential system \eqref{ODEsys}, has been extensively investigated in the literature. Different conditions are proposed to validate the MM kinetic, for instance,  when small initial enzyme concentrations are, or when fast reaction rate constants are large, see e.g. \cite{segel1989quasi}. It is also worth noting that this research direction belongs to a developed theory of fast-slow systems or multi-time scale dynamics (see e.g. \cite{kuehn2015multiple}).

\medskip
The singular limit as $\eps\to 0$ of \eqref{MainProblem} falls into the problem of fast reaction limits for PDE, which has caught a lot of attention in the past decades. Studies in this direction go back to the  works of Evans \cite{evans1980convergence} and Martin \cite{martin1980mathematical} in the eighties where the former showed the convergence to a nonlinear diffusion problem while the latter proved the convergence to a Stefan free boundary type problem. Already these
two works suggested interesting mathematical structures as well as complexity of fast reaction limits of PDE. Indeed, extensive studies on the subject showed that fast reaction limit leads to many different and interesting
limiting systems, ranging from nonlinear diffusion equations \cite{bothe2003reaction,evans1980convergence}, Stefan free boundary problems \cite{bothe2012instantaneous,murakawa2011fast}, cross-diffusion systems \cite{bothe2012cross,desvillettes2015new,daus2020cross,brocchieri2021evolution}, to a new derivation of the classical dynamical boundary conditions
\cite{henneke2016fast}, various behaviour of moving interfaces \cite{iida2017vanishing}, fractional kinetics \cite{conforto2018reaction,desvillettes2019non}, or systems involving Young measures \cite{perthame2022fast}.  On the one hand, this variety of limiting dynamics
shows the close connection of fast reaction limits of (bio-)chemical models to other phenomena of
dynamical systems depending on different scales and situations. For instance, the classical dynamical
boundary condition for parabolic equations, which has its root in modelling heat conduction in solids
(when a solid is in contact with a well-stirred fluid at its surface),   can be rigorously derived as
a limit of a volume-surface reaction-diffusion system in which the reaction rate between the volume and surface-concentrations tends to infinity \cite{henneke2016fast}. Another example is the famous SKT cross-diffusion
system, named after Shigesada, Kawasaki, and Teramoto \cite{shigesada1979spatial}, can be derived as a formal limit of
a reaction-diffusion system \cite{iida2006diffusion}. On the other hand, the analysis of reduced models  can
benefit from viewing it as the limit of a model which possesses useful structures. For instance, in \cite{daus2020cross}
a study of fast reaction limit leads to a reduced algebraic-cross-diffusion system whose analysis seem impenetrable
at the first glance, but becomes feasible thanks to the entropic structure of the original system, which
is propagated via the fast reaction limit.

\medskip
Our present paper contributes to this literature by showing that the classical MM kinetics for enzyme reaction in the presence of diffusion can also be rigorously derived from mass action kinetics by studying a fast reaction limit type problem using a suitable rescaling. We expect that our results have considerable applications, especially in deriving MM kinetics for (bio-)chemical catalytic reactions.

\subsection{Main results and key ideas}
From the reduced system \eqref{ReducedSys}, it is clear that the cases $d_2 = d_3$ and $d_2 \ne d_3$ lead to different limit systems. These also introduce different difficulties when showing the rigorous derivation of the MM kinetics. In the following, we write $L^{p+}$ and $W^{k,p+}$ to indicate $L^{p+\delta}$ or $W^{k,p+\delta}$ for some $\delta>0$. 

\begin{theorem}[The case $d_2 = d_3$]\label{thm1}
	Assume   $d_2 = d_3$. Consider componentwise non-negative initial data {$u_0 \in W^{2,q_0+}(\Omega) \times L^{q_0+}(\Omega)^2 \times W^{2,q_0+}(\Omega)$ for $q_0\ge \max\{N+2,4\}$}, and let $\ue$ be the classical solution to \eqref{MainProblem} for $\eps>0$.   Then we have, as $\eps\to 0$,
	\begin{equation*}
		(\ue_1, \ue_2, \ue_3, \ue_4) \longrightarrow (u_1, u_2,u_3, u_4) \quad \text{ in } \quad \LQ{\infty}\times \LQ{q_0+}\times \LQ{q_0+} \times \LQ{\infty},
	\end{equation*}
where $(u_1,v,u_4)$ with $v=u_2+u_3$ is the bounded weak solution to \eqref{ReducedSys} for $d_2 = d_3$. Moreover, we have the following convergence of the critical manifold
	\begin{equation}\label{f4b}
		\left\|   u_3^\epsilon - \frac{k_1 u_1^\epsilon + l_2u_4^\epsilon }{k_1u_1^\epsilon + l_2u_4^\epsilon +k_{2}+ l_{1}} v^\epsilon   \right\|_{L^{2}(Q_T)}    = O\left( \epsilon^{1/2}\right) \quad \text{as}\quad \eps\to 0 .
	\end{equation} 
\end{theorem}
\begin{remark} \label{FirstRemark}\hfill
	\begin{itemize}
		\item Due to the rescaling \eqref{MainScaling}, the assumption $u_{20}, u_{30}\in L^{q_0+}(\Omega)$ in Theorem \ref{thm1} is equivalent to $\|n_{EI}\|_{\LO{q_0+}} = O(\eps) = \|n_{CI}\|_{\LO{q_0+}}$. This is somewhat stronger than \eqref{a1}, where only the smallness in $L^1(\Omega)$ is assumed. Nevertheless, since $q_0<+\infty$, the assumption in Theorem \ref{thm1} still allows initial enzyme (and complex) to have spikes, which is biologically relevant (see \cite{FraLaxWalWit18}).
		
		\item The convergence of critical manifold in \eqref{f4b} can be shown in a better norm, namely $\LQ{p}$ for $p>2$, with a price of slower convergence order.
	\end{itemize}
\end{remark}
Thanks to $d_2 = d_3$, which means that $\lim_{\eps\to 0}|\de_2-\de_3| = 0$, and assumption \eqref{AssumpOnDj}, we can apply the improved duality lemma to see that $\{\ue_2\}_{\eps>0}$ and $\{\ue_3\}_{\eps>0}$ are bounded in $\LQ{q}$ for any $1\le q<\infty$.  From this, we can utilise the equations of $\ue_1$ and $\ue_4$ to show that $\{\ue_1\}_{\eps>0}$ and $\{\ue_4\}_{\eps>0}$ are relatively compact in $\LQ{\infty}$. Moreover, it can be shown that $\{\na \ue_j\}$, $j\in \{1,4\}$ are also bounded uniformly in $\LQ{\infty}$. In order to obtain the strong convergence of $\ue_2$ and $\ue_3$, we use an energy function of the form
\begin{equation}\label{energy}
	\mathcal{H}^\eps(t):= \int_{\Omega}\Big((k_1\ue_1 + l_2\ue_4)^{p-1}(\ue_2)^{p} + (k_2+l_1)^{p-1}(\ue_3)^p\Big)
\end{equation}
for some $p\ge 2$. The essential difficulty in our problem, comparing to that of \cite{desvillettes2015new,brocchieri2021evolution}, is that no positive lower bounds for $\ue_1$ or $\ue_4$ available. We therefore exploit the gradient estimates of $\ue_1, \ue_4$ and the $\LQ{q}$-estimates of $\ue_2, \ue_3$, together with a uniform bound of $\int_0^T\int_{\Omega}|\na \ue_j|^{2}/(\ue_j)^{1+\sigma}$, $j\in \{1,4\}$ for any $\sigma \in [0,1)$, to obtain strong convergence of the critical manifold \eqref{f4b} and uniform boundedness of $\{\na \ue_3\}_{\eps>0}$ in $\LQ{2}$. These firstly lead strong convergence of $\ue_2+\ue_3 \to u_2 + u_3$, and consequently $\ue_2 \to u_2$ and $\ue_3 \to u_3$ by combining with the strong convergence of the critical manifold, where $(u_1,v=u_2+u_3,u_4)$ is a weak solution to the reduced system \eqref{ReducedSys}. Furthermore, since $d_2 = d_3$, the limit system has a unique bounded weak solution, which implies that the whole sequence $\{\ue_j\}_{\eps>0}$, $j=1,\ldots, 4$, is convergent as $\eps\to 0$ rather than just a subsequence. 

\begin{theorem}[The case $d_2\ne d_3$, strong convergence of solutions]\label{thm2}
	Assume \eqref{AssumpOnDj} and $d_2 \ne d_3$. Assume additionally that there exists 
	\begin{equation}\label{p0_assump_thm2}
		p_0 > \left\{
		\begin{aligned}
			&4 && \text{ if } N = 1,2,\\
			&\frac{6(N+2)}{N+4} &&\text{ if } 3\le N\le 8,\\
			&\frac{8(N+2)}{N+8} &&\text{ if } N \ge 9 ,
		\end{aligned}
		\right.
	\end{equation}
	such that
	\begin{equation}\label{eq7}
		\frac{|d_2-d_3|}{d_2+d_3}<\frac{1}{C_{p_0'}^{\MR}} \quad \text{ with } \quad p_0' = \frac{p_0}{p_0-1}
	\end{equation}
	where $C_{p_0'}^{\MR}$ is the optimal constant in $L^p$-maximal regularity of parabolic equations (see Lemma \ref{LambertonLemma}). Consider componentwise non-negative initial data $u_0 \in W^{2,q_0+}(\Omega)\times \LO{q_0+}^2 \times W^{2,q_0+}(\Omega)$ for   {$q_0:=\max\{N,p_0,(N+2)/2\}$}, and let $\ue$ be the  classical solution to \eqref{MainProblem}.  Then we have, up to a subsequence as $\eps \to 0$,
	\begin{equation*}
		{ (\ue_1,\ue_2,\ue_3, \ue_4) \longrightarrow (u_1,u_2,u_3,u_4) \quad \text{ in } \LQ{p} \times \LQ{p_0-}^2 \times \LQ{p} }
	\end{equation*}
	for any {$1\le p < p_1$},
	
\begin{equation} 
	p_1 = \begin{cases}
			\frac{(N+2)p_0}{N+2-2p_0} &\text{ if } p_0 < (N+2)/2,\\
			< \infty \text{ arbitrary } &\text{ if } p_0 = (N+2)/2,\\
			 \infty &\text{ if } p_0 > (N+2)/2,
		\end{cases} 
	\end{equation}	
	 where $(u_1,v,u_4)$ with $v=u_2+u_3$ is a weak solution to \eqref{ReducedSys} (see Definition \ref{very_weak_sol} (a)). We also have the convergence of the critical manifold
	\begin{equation*}
		\left\|   u_3^\epsilon - \frac{(k_1 u_1^\epsilon + l_2u_4^\epsilon)v^\epsilon  }{k_1u_1^\epsilon + l_2u_4^\epsilon +k_{2}+ l_{1}}   \right\|_{L^{2}(Q_T)}    = O\left( \epsilon^{1/2}\right) \quad \text{as}\quad \eps\to 0.	
	\end{equation*}
\end{theorem}
\begin{remark}
	In Theorem \ref{thm2}, the strong convergence of $\ue_1 \to u_1$ and $\ue_4 \to u_4$ can be improved as
	\begin{equation*}
		(\ue_1,\ue_4) \to (u_1,u_4) \quad \text{ in } \quad \LQ{p^*}
	\end{equation*}
	{for any $p^* < +\infty$ if $N\le 8$, and  $p^* = 
	\frac{4(N+2)}{N-4}$ if $N\ge 9$}. Similarly as Remark \ref{FirstRemark}, the convergence of critical manifold can be obtained in a stronger norm $\LQ{s}$ for some $s>2$, especially in the lower dimensions.
\end{remark}
Theorem \ref{thm2} considers the case when $d_2 \ne d_3$, and consequently the limit system \eqref{ReducedSys} is a cross-reaction-diffusion system, and $d_2$ and $d_3$ close to each other enough such that \eqref{eq7} and \eqref{p0_assump_thm2} hold. This closeness condition is enough to apply the improved duality method to obtain a-priori bounds of $\ue_j$, $j=\{1,2,3,4\}$ to estimate the energy \eqref{energy}, which are then sufficient to deduce the desired convergences in   Theorem \ref{thm2}.

It can be seen from \eqref{p0_assump_thm2} that even in one dimension, conditions on $d_2$ and $d_3$ are still imposed since the unconditional improved duality method only gives $\LQ{2+}$ estimates. It is in fact possible to weaken \eqref{p0_assump_thm2} with the price of having only convergence for enzyme and complex.

\begin{theorem}[The case $d_2\ne d_3$, strong convergence of critical manifold]\label{thm3}
	Assume \eqref{AssumpOnDj} and $d_2 \ne d_3$. If $N\ge 3$, assume additionally that there exists 
	\begin{equation}\label{p0_assump_thm3}
		p_0 > \frac{3(N+2)}{N+4}
	\end{equation} 
	such that
	\begin{equation}\label{e0}
		\frac{|d_2-d_3|}{d_2+d_3} < \frac{1}{C_{p_0'}^{\MR}} \quad \text{ with } \quad p_0' = \frac{p_0}{p_0-1}
	\end{equation}
	where $C_{p_0'}^{\MR}$ is the optimal constant in $L^p$-maximal regularity of parabolic equations (see Lemma \ref{LambertonLemma}). Consider componentwise non-negative  initial data $u_0\in W^{2,q_0+}(\Omega)\times \LO{q_0+}^2 \times W^{2,q_0+}(\Omega)$ for   {$q:=\max\{N,p_0,(N+2)/2\}$}, and let $\ue$ be the classical solution to \eqref{MainProblem}. Then we have, up to a subsequence as $\eps \to 0$,
	\begin{equation*}
		(\ue_1,\ue_4) \longrightarrow (u_1,u_4) \;\text{ in } \; \LQ{p}^2, \quad \text{ and } \quad (\ue_2,\ue_3) \rightharpoonup (u_2,u_3) \; \text{ in } \; \LQ{2}^2
	\end{equation*}
	with $1\le p<\overline{p}_1$,
\begin{align*}
\overline{p}_1 = \left\{ \begin{array}{llll}
 + \infty & \text{if } N=1,2, \vspace*{0.15cm}\\
\displaystyle  \frac{3(N+2)}{N-2}  & \text{if } N\ge 3,	
\end{array}		 \right.
\end{align*}	
 where $(u_1,v = u_2+u_3,u_4)$ is a very weak solution to the limit  system \eqref{ReducedSys} (see Definition \ref{very_weak_sol} (b)). The following strong convergence of the critical manifold holds
	\begin{equation*}
		\left\|\ue_3 - \frac{(k_1\ue_1 + l_2\ue_4)v^\eps}{k_1\ue_1 + l_2\ue_4 + k_2 + l_1}\right\|_{\LQ{r}} = O(\eps^{1/6}).
	\end{equation*}
	for $r=4/3$ if $N=1,2,$ and $r = 6/5$ if $N\ge 3$.
\end{theorem}
Theorem \ref{thm3} allows us to remove the closeness of diffusion coefficients of $u_2, u_3$ in case $N=1,2$, and weaken it from \eqref{p0_assump_thm2} to \eqref{p0_assump_thm3} in higher dimensions. Comparing to Theorem \ref{thm2}, we also obtain the strong convergence of the critical manifold in Theorem \ref{thm3} but only weak convergence of $\ue_2, \ue_3$. It will become evident in our proof that under the assumptions of Theorem \ref{thm3}, it seems not possible to exploit the energy function \eqref{energy} due to the lack of suitable estimates for $\ue_j, j=1,2,3,4$. Our idea is to consider a \textit{modified energy function}
\begin{equation}\label{modified_energy}
	\mathcal{H}_\alpha^{\eps}(t):=  \int_{\Omega}\Big((k_1\ue_1 + l_2\ue_4 + \alpha(\eps))^{p-1}(\ue_2)^{p} + (k_2+l_1)^{p-1}(\ue_3)^p\Big),
\end{equation}
for $1<p<2$, where $\alpha(\eps)$ satisfies $\lim_{\eps\to 0}\alpha(\eps) = 0$. It turns out that for suitable $\alpha(\eps)$, we obtain the strong convergence of the critical manifold in Theorem \ref{thm3}.

\medskip
Finally, without imposing additional assumptions on the limit diffusion coefficients $d_1, \ldots, d_4$, we can still show the convergence of \eqref{MainProblem} to \eqref{ReducedSys} and the critical manifold in a weak sense. This is stated in our final result.
\begin{theorem}[The case $d_2 \ne d_3$, weak convergence of critical manifold]\label{thm4}
	Assume \eqref{AssumpOnDj} and $d_2 \ne d_3$. Consider componentwise non-negative initial data $u_0 \in \LO{2}^4$, and let $\ue$ be a global weak solution to \eqref{MainProblem}.   Then we have, up to a subsequence as $\eps \to 0$,
	\begin{equation*}
		(\ue_1,\ue_4)\to (u_1,u_4) \; \text{ in } \; \LQ{p}^2, \quad \text{ and } \quad (\ue_2,\ue_3) \rightharpoonup (u_2,u_3) \; \text{ in }\;\LQ{2}^2
	\end{equation*}
	with $1\le p<\overline{p}_1$ ($\overline{p}_1$ is defined in Theorem \ref{thm3}), where $(u_1,v = u_2+u_3,u_4)$ is a very weak solution to the reduced system \eqref{ReducedSys}. Moreover, the critical manifold converges to zero in distributional sense, i.e.
	\begin{equation}\label{f1}
		\lim_{\eps\to 0}\left|\intQT \sbra{\ue_3 - \frac{(k_1\ue_1 + l_2\ue_4)v^\eps }{k_1\ue_1 + l_2\ue_4 + k_2 + l_1}}\psi dxdt\right| = 0
	\end{equation}
	for any test function $\psi\in C_c^\infty(Q_T)$.
\end{theorem}

The results in our paper can be summarised in Figure \ref{fig:results}. In this table, we write $L^p:=\LQ{p}$, $L^{p-} = \cap_{1\le q<p}\LQ{q}$ and dist. sense for distributional sense.
\begin{figure}[ht]
	\begin{center}
		{\begin{tabular}{|c|c|c|c|c|}
			\hline

			\textbf{Rigorous derivation}& \multirow{2}{*}{$d_2 = d_3$} & \multicolumn{3}{c|}{$d_2\ne d_3$}\\
			\cline{3-5}
			\textbf{of MM kinetics}& &assume \eqref{p0_assump_thm2}-\eqref{eq7} & assume \eqref{p0_assump_thm3}-\eqref{e0} & unconditional\\
			\hline 
 			$\ue_j \to u_j$, $j\in\{1,4\}$ & $L^{\infty}$& $L^{p_1-}$ & $L^{\overline{p}_1-}$ & $L^{\overline{p}_1-}$ \\
			\hline
			$\ue_j \to u_j$, $j\in \{2,3\}$ &$L^{q_0+}$ & $L^{p_0-}$ &$weak-L^{2}$ & $weak-L^{2}$\\
			\hline
			conv. critical manifold & $O(\eps^{1/2})-L^{2}$ & $O(\eps^{1/2})-L^{2}$ & $O(\eps^{1/6})-L^{r}$ & {\it dist. sense}\\
			\hline
		\end{tabular}}
	\end{center}
	\caption{Rigorous derivation of Michaelis-Menten kinetics in the presence of diffusion.}
	\label{fig:results}
\end{figure}

\medskip
{\bf{Notation}}.
In this paper, we will use the following notation:
\begin{itemize}
	\item For $T>0$ we write $Q_T = \Omega\times(0,T)$. The classical Lebesgue spaces are denoted by $L^p(\Omega)$ and $\LQ{p}$, $1\le p \le \infty$. For $1\le p <\infty$, we write $u\in \LQ{p+}$ if there exists a constant $\gamma>0$ such that $u\in \LQ{p+\gamma}$.
	\item We denote by $C$ a generic constant which can be different from line to line or even in the same line. This constant may depend on fixed parameters of the problem, such as the dimension $N$, the domain $\Omega$, the fixed time horizon $T>0$, the limit diffusion coefficients $d_1,\ldots, d_4$ in \eqref{AssumpOnDj}, etc., but \textit{does not depend on the parameter $\eps>0$.} Sometimes we write $C(\alpha,\beta,\ldots)$ to emphasise the dependence of $C$ on the parameters $\alpha,\beta$, etc. 
\end{itemize}

\medskip
\textbf{Organisation of the paper}: In the next section, we derive uniform-in-$\eps$ bounds for the solution to \eqref{MainProblem}. These bounds, as mentioned earlier, are obtained by utilising the improved duality method, the heat regularisation, and a modified energy function, which are presented in the consecutive subsections. The last section is devoted to the proofs of the main theorems \ref{thm1}--\ref{thm4}.

\section{Uniform estimates} 

\subsection{Preliminaries}

We start with the global existence for the problem (\ref{MainProblem}) for each  positive value of the   parameter $\epsilon$.  We are interested in {classical solutions} in the following sense.

\begin{definition} Given $0< T \le  \infty$ and $\epsilon>0$. A vector   $u^\epsilon =(u_j^\epsilon )_{j=1,\ldots, 4}$ is called a classical   solution of  (\ref{MainProblem}) on $(0,T)$ if its components belong  to $  C([0,T);L^p(\Omega)) \cap C((\tau,T);L^\infty(\Omega))   \cap C^{2,1}(  \overline{\Omega}\times (0,T))$ for some $p>N/2$, and for all $T>\tau>0$,  and it satisfies (\ref{MainProblem}) pointwise.  
\end{definition}

\begin{theorem}\label{GloExis} Fix $\eps>0$. Then for each non-negative initial data $u_0 \in W^{2,q}(\Omega)\times \LO{q}^2\times W^{2,q}(\Omega)$, $q > \max\{N,2,(N+2)/2\}$,
	the system (\ref{MainProblem}) has a unique global classical solution $u^\epsilon$. 
\end{theorem}

\begin{proof} It is obvious that the nonlinearities $f_j: \R_+^4 \to \R$, for $j=1,\ldots, 4$, are   locally Lipschitz continuous and have at most quadratic growth. Therefore, by  standard fixed-point-arguments,  there exists only a local solution $u^\epsilon\in C([0,T_{\max}) ;L^p(\Omega))^4$, $N/2< p\le q$,  on a maximal interval $[0,T_{\max})$ to the   integral system 
\begin{align}
u_j^\epsilon(x,t) =e^{td_j\Delta}u_{j0}(x) + \int_0^t e^{(t-s)d_j\Delta} f_j(u^\epsilon(x,s))ds, \quad j=1,...,4, \label{SmoothingEffect} 
\end{align}
such that 
\begin{align}
T_{\max}=\infty \quad \text{or} \quad  \lim_{t\to T_{\max}^-} \|u(\cdot,t)\|_{L^{\infty}(\Omega)^4} = \infty \text{ if } T_{\max}<\infty. \label{TmaxExtend} 
\end{align}
The  solution preserves positivity since $f_j$, $j=1,...,4$, are  quasi-positive, i.e.
	$f_j(\ue)\ge 0$ for all $\ue\in [0,\infty)^4$ with $\ue_j = 0$. Moreover, thanks to   smoothing effects of the Neumann heat semigroup, $\ue_j(\tau)\in L^{\infty}(\Omega)$ for some $\tau\in (0,T_{\max})$. Taking these as initial data, and using the mass control structure
	\begin{equation*}
		\sum_{j=1}^4f_j(\ue) \le l_1\ue_3 + k_2\ue_3 \le \max\{l_1,k_2\}\sum_{j=1}^4\ue_j,
	\end{equation*}
	we can apply \cite{fellner2020global} to conclude that there exists a global classical solution to \eqref{MainProblem}.
%
%
%
\end{proof}

We now give a definition of weak and very weak solutions to \eqref{ReducedSys}.
\begin{definition}\label{very_weak_sol}
	Assume $d_2 \ne d_3$.
	\begin{itemize}
		\item[(a)] A triple of non-negative functions $(u_1,v,u_4)\in C([0,T];L^2(\Omega))^3\cap L^2(0,T;H^1(\Omega))^3$ is called a \textbf{weak solution} to \eqref{ReducedSys}, if 
		\begin{equation*}
			(\pa_t u_1, \pa_t v, \pa_t u_4) \in L^2(0,T;(H^1(\Omega))')\times L^2(0,T;(H^2(\Omega))')\times L^2(0,T;(H^1(\Omega))')
		\end{equation*}
		and for all test functions $\varphi \in L^2(0,T;H^1(\Omega))$,  $\psi \in L^2(0,T;H^2(\Omega))$ with {$\pa_{\nu}\psi = 0$} on $\pa\Omega\times(0,T)$ it holds
		\begin{equation*}
			\iint_{Q_T}\varphi\pa_t u_j  + d_j\iint_{Q_T}\na u_j \cdot \na \varphi = (-1)^j\iint_{Q_T}\frac{(k_1k_2u_1-l_1l_2u_4)v}{k_1u_1 + l_2u_4 + k_2 + l_1}\varphi, \quad j\in\{1,4\},
		\end{equation*}
		\begin{equation*}
			\iint_{Q_T}\psi \pa_t v + d_2\iint_{Q_T}\na v \cdot \na \psi = (d_3 - d_2)\iint_{Q_T}\frac{(k_1u_1+l_2u_4)v}{k_1u_1 + l_2u_4 + k_2 + l_1}\Delta \psi.
		\end{equation*}
		\item[(b)] A triple of non-negative functions $(u_1,v,u_4)$ is called a \textbf{very weak solution} to \eqref{ReducedSys}, if 
		\begin{equation*}
		(u_1,u_4)\in L^2(0,T;H^1(\Omega)), \quad (\pa_t u_1, \pa_t u_4) \in L^2(0,T;(H^1(\Omega))')^2, \quad v\in L^2(Q_T),
		\end{equation*}
		and for all test functions $\varphi\in L^2(0,T;H^1(\Omega))$, $\psi \in C_c^\infty(Q_T)$ it holds
		\begin{equation*}
		\iint_{Q_T}\varphi\pa_t u_j  + d_j\iint_{Q_T}\na u_j \cdot \na \varphi = (-1)^j\iint_{Q_T}\frac{(k_1k_2u_1-l_1l_2u_4)v}{k_1u_1 + l_2u_4 + k_2 + l_1}\varphi, \quad j\in\{1,4\},
		\end{equation*}
		\begin{equation*}
		-\iint_{Q_T}v \pa_t \psi - d_2\iint_{Q_T} v \Delta \psi =   (d_3 - d_2)\iint_{Q_T}\frac{(k_1u_1+l_2u_4)v}{k_1u_1 + l_2u_4 + k_2 + l_1}\Delta \psi.
		\end{equation*}
	\end{itemize}
\end{definition}

\subsection{Heat regularisation and improved duality method}
\label{HeatReguSubsec}

We first state the classical regularisation by the heat operator, which involves solutions to the following inhomogeneous heat equation   
\begin{equation}	\nonumber
(\H_d(u_0,f)): \quad 	\begin{cases}
		\partial_t u - d \Delta u = f, &\text{ in }  Q_T,\\
		\partial_{\nu}u = 0, &\text{ on } \pa\Omega\times(0,T),\\
		u(x,0) = u_0(x), &\text{ in } \Omega.
	\end{cases}
	\end{equation}

\begin{lemma}[\cite{lamberton1987equations}, Theorem 1]  \label{LambertonLemma}
	Let $1<q<\infty$. Assume that $f\in L^q(Q_T)$ and let $u$ be the weak solution to problem $(\H_1(0,f))$. Then there is an optimal constant $C_{q}^{\MR}$ depending only on $q$, the dimension $N$, and the domain $\Omega$, such that
	\begin{equation}\label{maxreg2}
		\|\Delta u\|_{L^q(Q_T)}\le C_{q}^{\MR}\|f\|_{L^q(Q_T)} , 
	\end{equation}
	where the superscript $\MR$ indicates the maximal regularity property.
	\end{lemma}
 
\begin{lemma} \label{lem:heat_regularisation}
	  Let $1<q<\infty$ and $d>0$. Assume that $f\in L^q(Q_T)$, $u_0\in W^{2,q}(\Omega)$, and $u$ be the {weak  solution} to the problem $(\H_d(u_0,f))$. Then, for 
	\begin{align*}
	p=\left\{ \begin{array}{lllllll}
	 \dfrac{(N+2)q}{N+2-2q}  & \text{if }  q<\frac{N+2}{2}, \vspace*{0.15cm}\\
	\in[1,\infty) \text{ arbitrary}  & \text{if }  q=\frac{N+2}{2}, \vspace*{0.15cm}\\
	\infty   & \text{if }  q>\frac{N+2}{2}, \vspace*{0.15cm}
	\end{array} 
	\right.
	\quad \text{and} \quad
	r=\left\{ \begin{array}{lllllll}
	\dfrac{(N+2)q}{N+2-q}  & \text{if }  q<N+2, \vspace*{0.15cm}\\
	\in[1,\infty)\text{ arbitrary}  & \text{if }  q=N+2, \vspace*{0.15cm}\\
	\infty   & \text{if }  q>N+2 , \vspace*{0.15cm}
	\end{array} 
	\right.
	\end{align*}
there holds 
\begin{align}
	&\|u\|_{L^{p}(Q_T)} + \|\nabla u\|_{L^{r}(Q_T)} + \|\partial_t u\|_{L^{q}(Q_T)} + \|\Delta u\|_{L^{q}(Q_T)} \nonumber\\
	& \hspace*{3cm} \le C_{1}  C_{q,N,\Omega,T}       \|f\| _{L^q(Q_{T})} + C_{2} C_{q,N,\Omega,T}     \| u_0\|_{W^{2,q}(\Omega)}   , \label{LpLqRegularity}
\end{align} 
where $C_{q,N,\Omega,T}$ depends only on $q,N,\Omega, T$, continuously depends on $T$, and remains bounded for 	finite values of $T>0$, and  
$$C_1:=  (C_{q}^{\MR} + 1  )   (1/d+1+T), \quad C_2:= T^{1/q} (2+d + dT).$$
\end{lemma}

\begin{proof} By the substitutions $s =dt$, $\widetilde{u}(s,\cdot) =u(s/d,\cdot)$, $\widetilde{f}(s,\cdot) =f(s/d,\cdot)$, the equation $(\H_d(u_0,f))$ of $u$  becomes the equation $(\H_1(u_0,\widetilde f/d))$ of $\widetilde u$. Let us consider the decomposition $\widetilde u=\widetilde v+\widetilde w$, where $\widetilde v$ and $\widetilde w$ are the solutions to  $(\H_1(0,\widetilde f/d))$ and $(\H_1(u_0,0))$, respectively.  
By Lemma \ref{LambertonLemma},   
\begin{align}
\|\Delta \widetilde  v\|_{L^q(Q_{dT})}\le  C_{q}^{\MR}  \| \widetilde f/d  \| _{L^q(Q_{dT})}   ,\label{HeatRegu1}
\end{align}   
 On the other hand, since the heat semigroup corresponding to the homogeneous Neumann boundary condition is a contraction semigroup on $L^q(\Omega)$, we have 
\begin{align}
  \|\Delta \widetilde w(s)\|_{L^q(\Omega)} =   \| \Delta e^{s\Delta}  u_{ 0}\|_{L^q(\Omega)} =   \| e^{s\Delta} \Delta u_{ 0}\|_{L^q(\Omega)} \le   \|\Delta u_{0}\|_{L^{q}(\Omega)},    \label{HeatRegu2}
\end{align}
for all $s\in(0,dT)$, where we note that $e^{s\Delta}$ and  $\Delta$ are commute on $W^{2,q}(\Omega)$   since $\{e^{s\Delta}:s\ge 0\}$ is a continuous semigroup.  
  Combining the estimates (\ref{HeatRegu1}) and (\ref{HeatRegu2}) gives 
\begin{align}
\|\Delta \widetilde u\|_{L^q(Q_{dT})}\le  C_{q}^{\MR}   \| \widetilde f/d \|_{L^q(Q_{dT})} +  (dT)^{\frac{1}{q}}   \|\Delta u_{0}\|_{L^{q}(\Omega)}  , \nonumber 
\end{align}   
which is equivalent to 
\begin{align}
  \| \Delta u\|_{L^q(Q_{T})}\le\,&\,  C_{q}^{\MR}    \| f/d \|_{L^q(Q_{T})} +   T^{\frac{1}{q}}   \|\Delta u_{0}\|_{L^{q}(\Omega)} .     
  \nonumber 
\end{align} 
and so the equation $(\H_d(u_0,f))$ gives     
$  \| \partial_t u\|_{L^q(Q_{T})}\le    (C_{q}^{\MR} + 1  )    \| f  \|_{L^q(Q_{T})} +   dT^{\frac{1}{q}}   \|\Delta u_{0}\|_{L^{q}(\Omega)} .   $ 
 Therefore, by the H\"older inequality,
\begin{align*}
\|u\|_{L^q(Q_T)} \le T^{\frac{1}{q}}\|u_0\|_{L^q(\Omega)} +  T \| \partial_t u \|_{L^q(Q_{T})},     
\end{align*} 
and consequently   $u\in W^{2,1}_q(Q_T)$ with respect to 
\begin{align}  
\|u\|_{W^{2,1}_q(Q_{T})}\le\,&\,        (C_{q}^{\MR} + 1  )   (1/d+1+T) \|f\| _{L^q(Q_{T})}  +  T^{1/q} (2+d + dT) \| u_0\|_{W^{2,q}(\Omega)}    . \label{HeatRegu3}
\end{align}
  Finally, by applying interpolation inequalities with different spaces of functions (on $\mathbb{R}^{N+1}$)  depending on $x$ and  $t$ as in  \cite[Lemma 3.3]{ladyvzenskaja1988linear}, there exists a constant $C_{q,N,\Omega,T}$  satisfying   (\ref{LpLqRegularity}). 
\end{proof}

\begin{remark} At the first glance, the above lemma looks similarly as \cite[Corollary of Theorem 9.1]{ladyvzenskaja1988linear}. 
We emphasise, however, that the regularity \eqref{LpLqRegularity} shows how the constants depend on $T$, and especially on  $d$. This is crucial to estimate solutions of reaction-diffusion equations with diffusion coefficients  $d_j^\epsilon$, $1\le j\le 4$, which depend on $\epsilon$. There are similar estimates as (\ref{LpLqRegularity}) in \cite[Corollary of Theorem 9.1]{ladyvzenskaja1988linear} and Ca\~nizo-Desvillettes-Fellner \cite[Lemma 3.3]{canizo2014improved}, where the dependence on $T$ was stated.
%
\end{remark}

We will utilise the following improved duality lemma.

\begin{lemma}[Improved duality estimate, \cite{canizo2014improved,einav2020indirect}] \label{ImprovedDuaEst} Let  $T>0$,  $1<q<\infty$, { $K\in \mathbb{R}$} and $q' = q/(q-1)$ be H\"older conjugate exponent of $q$. Assume that $X_1, \ldots, X_m$, $m\ge 2$, are nonnegative, smooth functions  satisfying the relation
	\begin{align}
	\left\{ \begin{array}{lllllll}
	\displaystyle \partial_t \bra{\sum_{i=1}^mX_i}  &\le&  \Delta \bra{\displaystyle\sum_{i=1}^m\kappa_iX_i} + {K} \displaystyle \sum_{i=1}^mX_i   & \text{in } Q_T, \vspace*{0.15cm}\\
	\partial_\nu  X_i &=& 0    & \text{on } \Gamma_T, \vspace*{0.15cm}\\
	 X_i(0,x) &=& X_{i,0}(x)    & \text{in } \Omega,
	\end{array} 
	\right.  \label{CrossDiffProblem}
	\end{align}
	for some constants $\kappa_i > 0$, $i=1,\ldots,m$.  Let $\kappa_{\min} = \min\{\kappa_i\}$ and $\kappa_{\max} = \max\{\kappa_i\}$. If  
	\begin{align}
	C_{q'}^{\MR} \frac{|\kappa_{\max}-\kappa_{\min}|}{\kappa_{\max}+\kappa_{\min}}  < 1 ,  \label{SmallDifferenced2d3}
	\end{align}
	and $\displaystyle \sum_{i=1}^mX_{i,0} \in L^q(\Omega)$, then  
	\begin{align}
	\sum_{i=1}^m\left\|X_{i}\right\|_{L^q(Q_T)} 
	\le C(T) \left\| \sum_{i=1}^m  X_{i,0} \right\|_{L^q(\Omega)}, 
\label{ImproveEstimate} 
	\end{align}
	where $C$ depends continuously on $T$ and on $\kappa_{\min}$, $\kappa_{\max}$.
\end{lemma}
\begin{proof}
	With the change of variable $\wt{X}_i(x,t) = e^{-Kt}X_i(x,t)$ we have
	\begin{equation*}
	\pa_t\bra{\sum_{i=1}^m\wt{X}_i}\le \Delta\bra{\sum_{i=1}^m \kappa_i\wt{X}_i}.
	\end{equation*}
	This allows us to apply the same arguments in \cite[Lemma 3.9]{einav2020indirect} to obtain, under the condition \eqref{SmallDifferenced2d3},
	\begin{equation*}
		\sum_{i=1}^m\|\wt{X}_i\|_{\LQ{q}} \le C(T)\left\|\sum_{i=1}^m\wt{X}_i(\cdot,0)\right\|_{\LO{q}}
	\end{equation*}
	which implies the desired estimates for $X_i$ in \eqref{ImproveEstimate}. Here we emphasise the continuous dependence of the constant $C$ on $\kappa_{\max}$ and $\kappa_{\min}$ as they are needed afterwards.
\end{proof}

\subsection{A modified energy function} \label{ModEneSection}
By observing  from the equations of $u_2^\epsilon$ and $u_3^\epsilon$ in (\ref{MainProblem}), it is natural to expect the convergence to the critical manifold, i.e.
$$\mathcal M(u^\epsilon):= (k_1 u_1^\epsilon + l_2   u_4^\epsilon)  u_2^\epsilon   -    (k_{2}+l_{1}) u_3^\epsilon     \to 0$$
in a suitable (preferably strong) topology. This turns out to be a subtle issue, especially in the case where the diffusion coefficients $d_2$ and $d_3$ are far away from each other. In the following, we look at the convergence of a perturbed critical manifold
\begin{align}
\widetilde{\mathcal M}_{\alpha}(u^\epsilon):= (k_1u_1^\epsilon + l_2u_4^\epsilon + \alpha(\eps)) u_2^\epsilon - (  k_{2}+l_{1} )  u_3^\epsilon  \to 0,   \label{MtoZero}
\end{align}
where $\alpha = \alpha(\epsilon)>0$ will be chosen later such that  $\lim_{\eps\to 0}\alpha(\eps)= 0$.  Obviously, the strong convergence of $\M(\ue)$ follows from the convergence of $\wt{\M}_\alpha(\ue)$ once $\ue_2$ is bounded in a certain norm. For the sake of simplicity, we  denote by
\begin{equation*}
	A_2  :=  k_1u_1^\epsilon + l_2u_4^\epsilon \quad \text{and}\quad A_3  :=    k_{2}+l_{1}.
\end{equation*}
Note that $A_2$ depends on $\eps>0$ while $A_3$ is a constant independent of $\eps$. For $t>0$, we   define the following function
\begin{align}
\mathcal H^\epsilon(t) :=   \int_\Omega      ((A_2+\alpha)u_2^\epsilon) ^{p-1}  u_2^\epsilon  + \int_\Omega  (A_3u_3^\epsilon) ^{p-1}   u_3^\epsilon =: \mathcal H_2^\epsilon(t) + \mathcal H_3^\epsilon(t), \quad  p>1 . \label{Hdef}
\end{align}
By differentiating this function with respect to the temporal variable, and employing the structure of solutions of   the problem (\ref{MainProblem}), we obtain a-priori estimates in the following lemma, which are crucial to derive the limit (\ref{MtoZero}) and also gradient estimates.

\begin{lemma} \label{TendToManifold} Let $\eps>0$ and $u^\epsilon$ be the solution of (\ref{MainProblem}).
	\begin{itemize}
		\item[(a)] Let $p\in [2,\infty)$ and $\alpha=0$.   Then, there exists  {$C=C(\|u_{0}\|_{L^\infty(\Omega) \times L^p(\Omega)^2 \times L^\infty(\Omega)})>0$}  independent of $\eps>0$ such that 
		\begin{equation}
		\label{EntropyInequalitya}		
		\begin{aligned}
		& \frac{1}{\epsilon}    \iint_{Q_T} \big|   A_2  u_2^\epsilon   -   A_3  u_3^\epsilon \big|^p   +    \sum_{j=2,3}   \iint_{Q_T} A_j^{p-1} (u_j^\epsilon)^{p-2} | \nabla u_j^\epsilon |^2        \\
		&    \hspace*{3cm} \le   C\bigg(1 +  \iint_{Q_T} (u_2^\epsilon)^p \Big[  A_2 ^{p-2}  \partial_t  A_2  +   A_2 ^{p-3}  |  \nabla \hspace*{-0.05cm}  A_2  |^2  \Big]   \bigg).  
		\end{aligned}
		\end{equation}		 
		
		\item[(b)] Let $p\in(1,2]$ and $\alpha=\epsilon^{\frac{1}{4-p}}$.  Then, there exists   {$C=C(\|u_{0}\|_{L^\infty(\Omega) \times L^p(\Omega)^2 \times L^\infty(\Omega)})>0$}  independent of $\eps>0$ such that 
		\begin{equation}
		\label{EntropyInequality}
		\begin{aligned}
		& \frac{1}{\epsilon^{\frac{1}{4-p}}}  \iint_{Q_T}   \frac{\big|   (A_2 + \epsilon^{\frac{1}{4-p}})  u_2^\epsilon   -   A_3  u_3^\epsilon  \big|^2}{\big( (A_2 + \epsilon^{\frac{1}{4-p}})  u_2^\epsilon     +  A_3  u_3^\epsilon\big)^{2-p}} 
		+ \epsilon^{\frac{3-p}{4-p}} \sum_{j=2,3} \iint_{Q_T}   A_j^{p-1} \frac{|\nabla u_j^\epsilon |^2}{(u_j^\epsilon)^{2-p}}      \\
		&  \hspace*{1.8cm} \le C \left(1 +  \iint_{Q_T} (u_2^\epsilon)^p \Big[   (A_2+1)^{p-1}   +          |\partial_t  A_2 | +  |  \nabla \hspace*{-0.05cm}  A_2  |^2  \Big]  \right). 
		\end{aligned}
		\end{equation}
	\end{itemize}

\end{lemma}

\begin{proof}  Let us firstly consider the   term $\mathcal H_2^\epsilon(t)$. The term $\mathcal H_3^\epsilon(t)$ can be treated similarly. By rewriting the equation of $u_2^\epsilon$ in (\ref{MainProblem}) as $\partial_t u_2^\epsilon - d_2^{\epsilon}\Delta u_2^\epsilon = - \mathcal M(u^\epsilon)/\epsilon$, we have, with the help of integration by parts,
	\begin{align}
	\frac{d\mathcal H_2^\epsilon}{dt} 
	&= - \frac{p}{\epsilon} \int_\Omega     \mathcal M (u^\epsilon) ((A_2 +\alpha)u_2^\epsilon)^{p-1}  +  (p-1) \int_\Omega (u_2^\epsilon)^p  (A_2 +\alpha) ^{p-2}  \partial_t  A_2 \nonumber \\
	&\quad + d_2^{\epsilon} p \int_\Omega (u_2^\epsilon)^{p-1}  (A_2 +\alpha) ^{p-1}      \Delta u_2^\epsilon  \nonumber\\
	&= -\frac{p}{\epsilon} \int_\Omega     \mathcal M (u^\epsilon) ((A_2 +\alpha)u_2^\epsilon)^{p-1} + (p-1) \int_\Omega (u_2^\epsilon)^p  (A_2 +\alpha) ^{p-2}  \partial_t  A_2   \nonumber
	\\ 
	&\quad - d_2^{\epsilon}p (p-1)  \int_\Omega   (u_2^\epsilon)^{p-1}  (A_2 +\alpha) ^{p-2}       \nabla  A_2  \nabla u_2^\epsilon      - d_2^{\epsilon}p (p-1) \int_\Omega (u_2^\epsilon)^{p-2}  (A_2 +\alpha) ^{p-1}   | \nabla     u_2^\epsilon |^2. \nonumber 
	\end{align} 
	For the term containing $ \nabla  A_2  \nabla u_2^\epsilon$ the elementary inequality $-xy\le x^2/2+y^2/2$ yields  
	\begin{align}
	\frac{d\mathcal H_2^\epsilon}{dt}  
	\le &  -  \frac{p}{\epsilon} \int_\Omega  \mathcal M (u^\epsilon) ((A_2 +\alpha)u_2^\epsilon)^{p-1}           -   \frac{d_2^{\epsilon}p(p-1)}{2}   \int_\Omega  (u_2^\epsilon)^{p-2} (A_2 +\alpha) ^{p-1}   | \nabla     u_2^\epsilon |^2           \nonumber\\
	&+ (p-1) \int_\Omega (u_2^\epsilon)^p  (A_2 +\alpha) ^{p-2}  \partial_t  A_2  + \frac{d_2^{\epsilon}p(p-1)}{2}  \int_\Omega   (u_2^\epsilon)^{p}  (A_2 +\alpha) ^{p-3}  |  \nabla \hspace*{-0.05cm}  A_2  |^2 .  \nonumber
	\end{align} 
	
	The derivative $d\mathcal H_3^\epsilon/dt$ can be estimated similarly, where we note that $\partial_t A_3$, $\nabla A_3$ are  equal to zero. 
	Thus, 
	by adding $d\mathcal H_2^\epsilon/dt$, $d\mathcal H_3^\epsilon/dt$,  and   integrating the resultant  over $(0,T)$, we obtain  
	\begin{equation} \label{EntroEstabA} 
	\begin{aligned}
	\mathcal H^\epsilon(T) +\,&\, \frac{1}{\epsilon}    \iint_{Q_T} \mathcal M (u^\epsilon) \Big( ( (A_2 +\alpha)  u_2^\epsilon)^{p-1}   - ( A_3  u_3^\epsilon)^{p-1} \Big)       +    \sum_{j=2,3}   \iint_{Q_T}  (u_j^\epsilon)^{p-2} A_j^{p-1}   | \nabla     u_j^\epsilon |^2             \\
	\,&\,     \le C \bigg(  \mathcal H (0) + \iint_{Q_T} (u_2^\epsilon)^p \Big[ (A_2 +\alpha) ^{p-2}  \partial_t  A_2  +      (A_2 +\alpha) ^{p-3}  |  \nabla \hspace*{-0.05cm}  A_2  |^2  \Big]  \bigg) ,  
	\end{aligned}
	\end{equation}
	where  the constant  $C$ depends only on $p,d_2,d_3$. Here  one can easily find a such constant, which does not depend on $\epsilon$, by recalling that $d_j^\epsilon \to d_j$ for $j=2,3$. Moreover, it is clear that the term $\mathcal H (0)$ is bounded under the regularity $u_{10},u_{40}\in L^\infty(\Omega)$ and $u_{20}, u_{30} \in L^p(\Omega)$.
	
	\vspace*{0.15cm}
	
	Let us show part (a). Since  $\alpha=0$ and $p\ge 2$,  we can apply by   the elementary inequality $|x^\lambda-y^\lambda|\ge   |x-y|^{\lambda}$ for all $x,y\ge 0$, $\lambda\ge 1$, which deduces   
	\begin{align}
	\mathcal M (u^\epsilon) \big( (  A_2   u_2^\epsilon)^{p-1}   - ( A_3  u_3^\epsilon)^{p-1} \big) \ge  \big|  A_2  u_2^\epsilon   -  A_3  u_3^\epsilon  \big|^{p} . \label{EntroEstabB}
	\end{align}
	The inequality (\ref{EntropyInequalitya}) follows from combining (\ref{EntroEstabA}), (\ref{EntroEstabB}) with the non-negativity of   $\mathcal H^\epsilon(T)$.
	
	\medskip
	
	For part (b), we plug  $\mathcal M(u^\epsilon)=\widetilde{\mathcal M}(u^\epsilon)-\alpha  u_2^\epsilon$ into (\ref{EntroEstabA}) to get
	\begin{align}
	&    \iint_{Q_T} \widetilde{\mathcal M} (u^\epsilon) \Big( ( (A_2 + \alpha)  u_2^\epsilon)^{p-1}   - ( A_3  u_3^\epsilon)^{p-1} \Big)    + \epsilon \sum_{j=2,3} \iint_{Q_T}  (u_j^\epsilon)^{p-2} A_j^{p-1}   | \nabla     u_j^\epsilon |^2   \nonumber\\
	&  \le C \left( \epsilon \mathcal H ^\epsilon(0)  +    \iint_{Q_T} (u_2^\epsilon)^p \Big[ \alpha  (A_2 + \alpha)^{p-1} +  \epsilon (A_2 + \alpha) ^{p-2}  \partial_t  A_2  +     \epsilon (A_2 + \alpha) ^{p-3}  |  \nabla \hspace*{-0.05cm}  A_2  |^2 \Big] \right) \nonumber\\
	&  \le C \left(\eps \mathcal H ^\epsilon(0)  +    \iint_{Q_T} (u_2^\epsilon)^p \Big[  \alpha  (A_2 + 1)^{p-1} +  \epsilon   \alpha  ^{p-2}  |\partial_t  A_2|  +     \epsilon   \alpha ^{p-3}  |  \nabla \hspace*{-0.05cm}  A_2  |^2 \Big] \right) \nonumber\\ 
	& \le  C \left(1 + \iint_{Q_T} (u_2^\epsilon)^p \Big[   (A_2+1)^{p-1}   +          |\partial_t  A_2 | +  |  \nabla \hspace*{-0.05cm}  A_2  |^2  \Big] \right) (\epsilon + \alpha + \epsilon \alpha^{p-2} + \epsilon \alpha^{p-3} ) ,    
	\nonumber 
	\end{align}
	where the constant  $C$ depends only on $p,d_2,d_3$.  Now, for all $x,y>0$ and $\lambda\in(0,1)$, there   exists  $z_{xy}$ such that  
	$|x^{\lambda-1}-y^{\lambda-1}| = (\lambda-1)z_{xy}^{\lambda-2}|x-y| $ and $\min(x,y)\le z_{xy}\le \max(x,y)$. Since $\lambda-2$ is negative, 
	$$|x^{\lambda-1}-y^{\lambda-1}| \ge  (\lambda-1)(x+y)^{\lambda-2}|x-y| .$$  By applying this inequality, we have  
	\begin{align*}
	\widetilde{\mathcal M} (u^\epsilon) \Big(  (A_2 + \alpha)  u_2^\epsilon   -  A_3  u_3^\epsilon  \Big)  \ge  (p-1)  \left( (A_2 + \alpha)  u_2^\epsilon     +  A_3  u_3^\epsilon\right)^{p-2}  \big| ( A_2 + \alpha)  u_2^\epsilon  -  A_3  u_3^\epsilon \big|^2.
	\end{align*}
	By choosing $ \alpha = \eps^\delta>0$ for $\delta>0$ and noting that
	\begin{align*}
		\max_{0<\delta < 1/(3-p)}  \min\big\{1;\,\delta;\,1-(2-p)\delta;\,1-(3-p)\delta \big\} = 1/(4-p),  
	\end{align*}
	we arrive at the optimal choice $\alpha=\epsilon^{\frac{1}{4-p}}$. This leads to the desired estimate \eqref{EntropyInequality}. 
\end{proof}

\begin{remark} 
		Lemma \ref{TendToManifold} suggests that we need to control terms on the right hand side of \eqref{EntropyInequality} and \eqref{EntropyInequalitya} uniformly in $\eps>0$. This certainly depends on uniform estimates of solutions to \eqref{MainProblem} that we derive. As it turns out, part (a) will be utilised when very good controls of solutions can be obtained, while part (b) is more suitable for having weaker controls of solutions. It is also remarked that the latter does not give us any uniform estimates for the gradients of $\ue_2$ and $\ue_3$.
\end{remark}

%

\subsection{Uniform-in-$\varepsilon$ bounds}

\begin{lemma}  \label{IrRegularityLem} 
   Then there exists $\epsilon_\ast>0$ such that
	\begin{equation}\label{L2plus}
 \sup_{0<\epsilon< \epsilon_\ast} \left(\sum_{j=1}^{4}\|\ue_j\|_{\LQ{2+}}\right) \le { C \left(T,\|u_0\|_{L^{2+}(\Omega)^4} \right)}. 
	\end{equation}
		 Moreover,
	\begin{itemize}
		\item[(a)] If $d_3 - d_2 = 0$, then for any $1<p<\infty$, there exists $\eps_p>0$ such that 
		\begin{equation*}
			\sup_{0<\epsilon< \epsilon_p}\bra{\|\ue_2\|_{\LQ{p}} + \|\ue_3\|_{\LQ{p}}} \leq {C\left(T,\|u_{20}\|_{L^p(\Omega)},\|u_{30}\|_{L^p(\Omega)}\right)}. 
		\end{equation*}
		 
		\item[(b)] If $d_3 - d_2 \not = 0$, then for any $2<p<\infty$, if $d_2$ and $d_3$ satisfy 	\begin{equation}\label{closeness_d2d3}
			C_{p'}^{\MR}\frac{|d_2 - d_3 |}{d_2 +d_3 } < 1,
		\end{equation}
		where $p' = p/(p-1)$ is the conjugate H\"older exponent of $p$, then 
there exists $\eps_p>0$ such that 		\begin{equation}\label{Lpplus}
			\sup_{0<\eps< \epsilon_p}\left(\|\ue_2\|_{\LQ{p}}+\|\ue_3\|_{\LQ{p}}\right) \le {C\left(T,\|u_{20}\|_{L^p(\Omega)},\|u_{30}\|_{L^p(\Omega)}\right)}. 
		\end{equation}
	\end{itemize} 
\end{lemma}

\begin{proof}
	Adding the equations in \eqref{MainProblem} leads to, thanks to the non-negativity of $\ue_j$, $j=1,\ldots, 4$,
	\begin{equation*}
		\left\{
		\begin{aligned}
			\partial_t \sum_{j=1}^4\ue_j - \Delta\bra{\sum_{j=1}^4d_j^{\epsilon}\ue_j} &\le (l_1+k_2)\sum_{j=1}^4\ue_j &&\text{ in } Q_T,\\
			\partial_{\nu}\ue_j &= 0, &&\text{ on } \Gamma_T,\\
			\ue_j(0) &= u_{j,0} &&\text{ in } \Omega.
		\end{aligned}
		\right.
	\end{equation*}
  According to \cite[Lemma 3.19]{einav2020indirect}, there exists a constant $1<p_\ast<2$ such that 
	\begin{equation*}
		C_{p_\ast}^{\MR}\frac{|\max_{j}d_j - \min_{j}d_j|}{\max_{j}d_j+\min_{j}d_j} < 1.
	\end{equation*}
Since $d_j^\epsilon \to d_j$ as $\epsilon\to 0^+$ for $j=1,\dots,4$, the exists $\epsilon_\ast>0$ such that 	
\begin{equation*}
	C_{p_\ast}^{\MR}\frac{|\max_{j}d_j^\epsilon - \min_{j}d_j^\epsilon|}{\max_{j}d_j^\epsilon+\min_{j}d_j^\epsilon}  < 1,  
	\end{equation*}
	for all $0<\epsilon< \epsilon_\ast$. 
Lemma \ref{ImprovedDuaEst} therefore yields that there exists $C_{\max_{j}d_j^\epsilon,\,\min_{j}d_j^\epsilon}(\|u_0\|_{L^{2+}(\Omega)^4})$ depending on $T,l_1,k_2,\Omega$ and  continuously depending on $\max_{j}d_j^\epsilon$,	$\min_{j}d_j^\epsilon$  such that 
\begin{equation} \nonumber  \sup_{0<\epsilon< \epsilon_\ast} \bigg(  \sum_{j=1}^{4}\|\ue_j\|_{\LQ{2+}} \bigg) \le   \left( \sup_{0<\epsilon< \epsilon_\ast} C_{\max_{j}d_j^\epsilon,\,\min_{j}d_j^\epsilon}  \right) \le    C(  \max_{j}d_j ,\,\min_{j}d_j ) , 
	\end{equation}	
which shows  \eqref{L2plus}. 

\medskip
In the case where $d_3 - d_2 = 0$, it follows from \eqref{AssumpOnDj} that the fraction $ |d_3^{\epsilon}-d_2^{\epsilon}|/(d_3^{\epsilon}+d_2^{\epsilon}) \to 0$. Therefore, for any $p\in (1,\infty)$, there exists $\eps_p>0$ such that
	\begin{equation}\label{eq1}
		C_{p'}^{\MR}\frac{|d_3^{\epsilon}-d_2^{\epsilon}|}{d_3^{\epsilon}+d_2^{\epsilon}} < 1,  
	\end{equation}
for all $0<\epsilon< \epsilon_p$. Consequently, Lemma \ref{ImprovedDuaEst} entails the uniform w.r.t. $\eps\in (0,\eps_p)$ bounds in $\LQ{p}$ of $\ue_2$ and $\ue_3$. 

\medskip
In the case $d_3 - d_2 \not = 0$, if \eqref{closeness_d2d3} is fulfilled, then due to \eqref{AssumpOnDj}, we can find $\eps_p > 0$ such that \eqref{eq1} holds. The estimate \eqref{Lpplus} is then a direct consequence of Lemma \ref{ImprovedDuaEst}.
\end{proof}
Without further assumptions, the improve duality method only gives us $\LQ{2+}$-estimates for $\ue_1$ and $\ue_4$. To obtain better estimates for these two functions, we utilise their equations  and the heat regularisation.
\begin{lemma}\label{lem:Estimates_u1u4}
	Let $\ue$ be the solution to \eqref{MainProblem}. We have the following uniform-in-$\eps$ bounds
	\begin{equation} \label{e1}
	\begin{aligned}
	& \sup_{0<\epsilon< \epsilon_\ast}  	\sum_{j\in\{1;4\}}\bra{\|\ue_j\|_{\LQ{\frac{2(N+2)}{N-2}+}}+\|\na \ue_j\|_{\LQ{2}} + \|\partial_t\ue_j\|_{\LQ{1+}} }   \\
	& \hspace*{2cm}   \le {C\left(T,\|u_{0}\|_{W^{2,2+}(\Omega) \times L^{2+}(\Omega)^2 \times W^{2,2+}(\Omega)}\right)},
	\end{aligned}
	\end{equation}
	where we use the convention $\frac{2(N+2)}{N-2}+ = +\infty$ if $N\le 2$.
\end{lemma}

\begin{proof}  Since the nonlinearities   $f_1(\ue) = -k_1\ue_1\ue_2 + l_1\ue_3$ and $f_4(\ue)=  -l_2 u_2^\epsilon  u_4^\epsilon  +  k_{2} u_3^\epsilon$ are quadratic, by Lemma \ref{IrRegularityLem} we have 
	\begin{equation}
	\sup_{0<\epsilon< \epsilon_\ast} \bigg(  \sum_{j\in\{1;4\}} \|f_j(\ue)\|_{\LQ{1+}}  \bigg) \le {C(T,\|u_0\|_{L^{2+}(\Omega)^4})} . \label{e1a}
	\end{equation} 	
Hence, from the equations 	\begin{equation*}
		\partial_t \ue_j - d_j^{\epsilon}\Delta \ue_j = f_j(\ue), \quad j=1,4,   
	\end{equation*}	
	we can now use the maximal regularity in Lemma \ref{lem:heat_regularisation} to get the bound on the derivative
	\begin{align}
	\sup_{0<\epsilon< \epsilon_\ast}    \sum_{j\in\{1;4\}} \|\partial_t\ue_j\|_{\LQ{1+}}  \le {C\left(T,\|u_{0}\|_{W^{2,2+}(\Omega) \times L^{2+}(\Omega)^2 \times W^{2,2+}(\Omega)}\right)}, \label{e1b}
	\end{align}  
where we have used (\ref{e1a}) and   $d_j^\epsilon \to d_j$ as $\epsilon\to 0^+$. Note that a direct application of Lemma \ref{lem:heat_regularisation} with the right hand side $f_j(\ue)\in \LQ{1+}$ does not give the uniform estimates for $\ue_j$ and $\na \ue_j$ as in \eqref{e1}, especially in high dimensions. We will utilise the non-negativity of $\ue$. Indeed, from
	\begin{equation*}
		\partial_t \ue_1 - d_1^{\epsilon}\Delta \ue_1 = -k_1\ue_1\ue_2+l_1u_3\le l_1\ue_3,
	\end{equation*}
	we see that $0\le \ue_1 \le w^\epsilon$ where $w^\epsilon$ is the solution to
	\begin{equation*}
		\partial_t w^\epsilon- d_1^{\epsilon}\Delta w^\epsilon = l_1\ue_3, \quad \partial_{\nu}w^\epsilon = 0, \quad w^\epsilon(x,0) = u_{10}(x) 
	\end{equation*}
in which $u_3^\epsilon$ is uniformly bounded in $L^{2+}(Q_T)$ w.r.t. $\epsilon$. Now we can apply Lemma \ref{lem:heat_regularisation} to the equation of $w^\epsilon$, and the comparison principle to obtain 
	\begin{equation*}
		\|\ue_1\|_{\LQ{\frac{2(N+2)}{N-2}+}} \le  \|w^\epsilon \|_{\LQ{\frac{2(N+2)}{N-2}+}} \le {C\left(T,\|u_{10}\|_{W^{2,2+}(\Omega)}\right)} .
	\end{equation*}
	The estimate for $\ue_4$ can be shown in the same way. Concerning the gradient estimate, we multiply the equation of $\ue_1$ by $\ue_1$ then integrate on $Q_T$ to get
	\begin{equation*}
		\frac 12\|\ue_1(T)\|_{\LO{2}}^2 + d_1^{\epsilon}\iint_{Q_T}|\na \ue_1|^2dxdt \le \frac 12\|u_{10}\|_{\LO{2}}^2 + l_1\iint_{Q_T}\ue_1\ue_3dxdt.
	\end{equation*}
	Thanks to the $\LQ{2+}$-bounds of $\ue_1$, $\ue_3$, and \eqref{AssumpOnDj} we get the gradient estimates of $\ue_1$. The estimate of $\na\ue_4$ follows in the same way, so we omit it.
\end{proof}

\begin{lemma} \label{IrRegularityLemU1}  Let $q $  and $\epsilon_*$ be defined by Lemma \ref{IrRegularityLem}. Then, for  all $ \sigma \in [0,1)$, 
	\begin{align}
		\sup_{0<\epsilon<\epsilon_*} \bigg(\sum_{j=1,4} d_j^{\epsilon} \iint_{Q_T} \frac{|\na \ue_j|^{2}}{(\ue_j)^{1+\sigma}}  +   \sum_{j=1,4}\iint_{Q_T}  \frac{ u_3^\epsilon}{(u_j^\epsilon)^{\sigma}}  \bigg) \le {C\left(T,\|u_0\|_{L^{2+}(\Omega)^4} \right)} . \label{EnerEstu1u4}
	\end{align}
\end{lemma}

\begin{proof} 
	For $\sigma \in[0,1)$ and $\delta>0$, we define 
	\begin{align}
		\mathcal E^{\delta,\epsilon}(t) := \int_\Omega (u_1^\epsilon(t)+\delta)^{1-\sigma} + \int_\Omega (u_4^\epsilon(t)+\delta)^{1-\sigma} =: \mathcal E_1^{\delta,\epsilon}(t)  + \mathcal E_4^{\delta,\epsilon}(t)  . \nonumber
	\end{align}
	We denote by $u_j^{\delta,\epsilon}:=u_j^\epsilon+\delta$ for $\delta>0$. By the first equation of \eqref{MainProblem},    
	\begin{align}
		\frac{d \mathcal E_1^{\delta,\epsilon}  }{dt}  =\,&\,  \frac{4\sigma}{1-\sigma} d_1^{\epsilon} \int_\Omega \Big|\nabla \sqrt{(u_1^{\delta,\epsilon})^{1-\sigma}} \hspace*{0.05cm} \Big|^2 + (1-\sigma) \int_\Omega (u_1^{\delta,\epsilon})^{-\sigma} (-k_1u_1^\epsilon u_2^\epsilon + l_1 u_3^\epsilon). \nonumber 
	\end{align}
	Integrating the latter equality over $(0,T)$ and noting the nonnegativity of $\mathcal E_1$ give
	\begin{equation}
	\begin{aligned}
		&\, \mathcal E_1^{\delta,\epsilon}(0)+ \frac{4\sigma}{1-\sigma} d_1^{\epsilon}  \iint_{Q_T} \Big|\nabla \sqrt{(u_1^{\delta,\epsilon})^{1-\sigma}} \hspace*{0.05cm} \Big|^2 +  l_1(1-\sigma)\iint_{Q_T}  \frac{ u_3^\epsilon}{(u_1^{\delta,\epsilon})^{\sigma}}   \\
		&\, \hspace*{1cm} = \mathcal E_1^{\delta,\epsilon}(T)  + (1-\sigma)k_1    \iint_{Q_T}     (u_1^{\delta,\epsilon})^{-\sigma}u_1^\epsilon  u_2^\epsilon .  
	\end{aligned}
	\end{equation}
	Summing the equations of $\ue_2$ and $\ue_3$ then integrate on $\Omega\times(0,T)$ lead to
	\begin{equation*}
		\sup_{T\ge 0}\int_{\Omega}(\ue_2(T)+\ue_3(T))dx \le \int_{\Omega}(u_{20}+u_{30})dx.
	\end{equation*}
	Integrating the equation of $\ue_1$ on $\Omega\times(0,T)$ gives
	\begin{align*}
		 \int_{\Omega}\ue_1(T)dx &\le \int_{\Omega}u_{10}dx + l_1\int_0^T\int_{\Omega}\ue_3(t)dxdt\\
		&\le \int_{\Omega}u_{10}dx + l_1T\int_{\Omega}(u_{20}+u_{30})dx.
	\end{align*}
	Therefore, by H\"older's inequality
	\begin{equation*}
		\mathcal{E}_{1}^{\delta,\eps}(T) \le \bra{\int_{\Omega}(\ue_1(T)+1)dx}^{1-\sigma}|\Omega|^{\sigma} \le C(T).
	\end{equation*}

	We now use $(u_1^{\delta,\eps})^{-\sigma}\ue_1 \le (u_1^{\delta,\eps})^{1-\sigma} \le u_1^{\delta,\eps} + 1$, let $\delta \to 0$, then use Fatou's lemma to get
	\begin{align*}	
		&\,   \frac{4\sigma}{1-\sigma} d_1^{\epsilon}  \iint_{Q_T} \Big|\nabla \sqrt{(u_1^{\epsilon})^{1-\sigma}} \hspace*{0.05cm} \Big|^2 +  l_1(1-\sigma)\iint_{Q_T}  \frac{ u_3^\epsilon}{(u_1^{ \epsilon})^{\sigma}}  \nonumber\\
		&\le C(T) +  k_1    \limsup_{\delta \to 0}\iint_{Q_T}     (u_1^{\delta,\epsilon})^{1-\sigma}  u_2^\epsilon   \nonumber \\
		&\le C(T) + k_1 \|u_1^{\delta,\epsilon}+1\|_{L^{2}(Q_T)}  \| u_2^\epsilon \|_{L^2(Q_T)}   \le {C\left(T,\|u_0\|_{L^{2+}(\Omega)^4} \right)}.  \nonumber  
	\end{align*}
	The term $\mathcal E_4^{\delta,\epsilon}$ can be treated similarly to $\mathcal E_1^{\delta,\epsilon}$. The inequality (\ref{EnerEstu1u4}) then follows.
\end{proof}

\section{Proofs}

\subsection{Proof of Theorem \ref{thm1}} 

\begin{lemma}  \label{IrRegularityLemCase1} Assume \eqref{AssumpOnDj} and $d_2 = d_3$. For any $q> N+2 $, there  exists  $\epsilon_q  \in (0,1)$ such that  
\begin{align}
 \sup_{0<\epsilon<\epsilon_q} \Big(  \|\partial_t u_1^\epsilon\|_{L^{q}(Q_T)} + \|\Delta u_1^\epsilon\|_{L^{q}(Q_T)} + \|\ue_1\|_{\LQ{\infty}} + \|\nabla \ue_1\|_{\LQ{\infty}}\Big) \le C(T),  \label{RegularityU1IrCase1} \\
  \sup_{0<\epsilon<\epsilon_q} \Big(  \|\partial_t u_4^\epsilon\|_{L^{q}(Q_T)} + \|\Delta u_4^\epsilon\|_{L^{q}(Q_T)} + \|\ue_4\|_{\LQ{\infty}} + \|\na\ue_4\|_{\LQ{\infty}} \Big) \le C(T), 
  \label{RegularityU4IrCase1}
\end{align} 
{where the constants depend on $\|u_0\|_{W^{2,q}(\Omega) \times L^q(\Omega)^2 \times W^{2,q}(\Omega)}$.}
\end{lemma}

\begin{proof} By Lemma \ref{IrRegularityLem} there  exists  $\epsilon_q  \in (0,1)$ such that 
\begin{align}
\sup_{0<\epsilon< \epsilon_q} \Big( \|u_2^\epsilon\|_{L^{q}(Q_T)  } + \|u_3^\epsilon\|_{L^{q}(Q_T) } \Big) \le {C\left(T,\|u_{20}\|_{L^q(\Omega)},\|u_{30}\|_{L^q(\Omega)}\right)}.  \label{RegularityU3VIrCase1}
\end{align}
From the equation of $\ue_1$ in \eqref{MainProblem}, $\pa_t u_1^\epsilon -  d_1 \Delta u_1^\epsilon   \le l_1 u_3^\epsilon$, where $l_1 u_3^\epsilon$ is uniformly bounded in $ L^q(Q_T)$ w.r.t. $\epsilon$.  
Since $q>  (N+2)/2 $, we can make use of the comparison principle and  heat regularisation (see Lemma \ref{lem:heat_regularisation})  to conclude that   
\begin{equation*}
	   \|\ue_1\|_{\LQ{\infty}}\le {C\left(d_1^\epsilon,T,\|u_{10}\|_{W^{2,q}(\Omega)}\right) \le C\left( T,\|u_{10}\|_{W^{2,q}(\Omega)}\right)} 
\end{equation*}
since $C (d_1^\epsilon,T,\|u_{10}\|_{W^{2,q}(\Omega)} )$ depends continuously on $d_1^\eps$ and due to \eqref{AssumpOnDj}. 
This, together with (\ref{RegularityU3VIrCase1}), implies that 
\begin{align*}
	\sup_{0<\epsilon<\epsilon_q} \left\| - k_1u_1^\epsilon u_2^\epsilon + l_1 u_3^\epsilon\right\|_{L^q(Q_T)} \le C(T).
\end{align*}
Taking into account $q>N+2$, another application of Lemma \ref{lem:heat_regularisation} gives the desired estimates
\begin{align*}
	\sup_{0<\epsilon<\epsilon_q} \left( \|\partial_t u_1^\epsilon\|_{\LQ{q}} +  \|\Delta u_1^\epsilon\|_{\LQ{q}} + \|\na\ue_1\|_{\LQ{\infty}} \right)  \le C(T) .  
\end{align*}
 Estimates for $\ue_4$ are obtained in the same way.
\end{proof}
The bounds in Lemma \ref{IrRegularityLemCase1} will be used in combination with the modified energy in Lemma \ref{TendToManifold} (a) to obtain convergence of critical manifold as well as gradient estimates of $\ue_2$ and $\ue_3$.

\begin{lemma}\label{lem1}
	Assume \eqref{AssumpOnDj} and $d_2 = d_3$. Then there exists $\eps_0>0$ such that
	\begin{equation}
	\begin{aligned}
		& \sup_{0<\eps<\eps_0}\bra{ \|\na \ue_3\|_{\LQ{2}}^2 +  \frac{1}{\eps}\iint_{Q_T}\left|(k_1\ue_1+l_2\ue_4)\ue_2 - (k_2+l_1)\ue_3\right|^2  }  \\
		& \hspace*{2.5cm} \le { C(T,\|u_0\|_{W^{2,q}(\Omega)\times L^{q}(\Omega)^2 \times W^{2,q}(\Omega)}) } ,
	\end{aligned}
	\end{equation}
	{where $q>\max\{N+2;4\}$.}
\end{lemma}
\begin{proof}
	From Lemma \ref{TendToManifold} (a), by choosing $p=2$, we have in particular
	\begin{align}\label{eq2}
		\frac{1}{\eps}\iint_{Q_T}\big|A_2\ue_2 - A_3\ue_3\big|^2 + A_3\iint_{Q_T}|\na \ue_3|^2 \le C\bra{1+\iint_{Q_T}(\ue_2)^2\sbra{\pa_t A_2 + \frac{|\na A_2|^2}{A_2}}},
	\end{align}
	{where $C=C(\|u_0\|_{L^{\infty}(\Omega) \times L^2(\Omega)^2 \times L^{\infty}(\Omega)})$.}
	From Lemmas \ref{IrRegularityLem} (a) and \ref{IrRegularityLemCase1},   there exists $0<\eps_0<\epsilon_q$ such that 
	\begin{equation}
	\label{eq3}
	\begin{aligned}
		\left|\iint_{Q_T}(\ue_2)^2\pa_tA_2\right| \le &\, C\|\ue_2\|_{\LQ{4}}^2\bra{\|\pa_t\ue_1\|_{\LQ{2}} + \|\pa_t\ue_4\|_{\LQ{2}}} \\
		\le &\, { C(T,\|u_0\|_{W^{2,q}(\Omega)\times L^{q}(\Omega)^2 \times W^{2,q}(\Omega)}) }.
	\end{aligned}
	\end{equation}
	Now using Lemmas \ref{IrRegularityLemCase1} and \ref{IrRegularityLem},  \ref{IrRegularityLemU1}, we estimate the remaining term for some $\sigma\in (0,1)$, $\sigma$ is enough closed to $1$,
	\begin{equation}\label{eq4}
	\begin{aligned}
		\left|\iint_{Q_T}(\ue_2)^2\frac{|\na A_2|^2}{A_2} \right| &= \iint_{Q_T}(\ue_2)^2\frac{|\na A_2|^{2/(1+\sigma)}}{A_2}|\na A_2|^{2\sigma/(1+\sigma)}\\
		&\le \|\na A_2\|_{\LQ{\infty}}^{\frac{2\sigma}{1+\sigma}}\bra{\iint_{Q_T}\frac{|\na A_2|^2}{A_2^{1+\sigma}}}^{\frac{1}{1+\sigma}}\bra{\iint_{Q_T}(\ue_2)^{\frac{2(1+\sigma)}{\sigma}}}^{\frac{\sigma}{1+\sigma}}\\
		&\le { C(T,\|u_0\|_{W^{2,q}(\Omega)\times L^{q}(\Omega)^2 \times W^{2,q}(\Omega)}), }
	\end{aligned}
	\end{equation}
	where we used
	\begin{equation*}
		\iint_{Q_T}\frac{|\na A_2|^2}{A_2^{1+\sigma}} \le C\bra{\iint_{Q_T}\frac{|\na \ue_1|^2}{(\ue_1)^{1+\sigma}} + \iint_{Q_T}\frac{|\na \ue_4|^2}{(\ue_4)^{1+\sigma}}} \le C(T)
	\end{equation*}
	thanks to the non-negativity of $\ue_1$ and $\ue_4$. From \eqref{eq3} and \eqref{eq4} we obtain the desired estimates of Lemma \ref{lem1}.
\end{proof}
Since $A_2 = k_1\ue_1 + l_2\ue_4$ and we do not have lower bounds of $\ue_1$ and $\ue_4$, the energy estimates in Lemma \ref{TendToManifold} (a) does not give gradient estimate for $\ue_2$. To overcome this, we use the relation between $\ue_2$ and $\ue_3$
\begin{equation}\label{eq5}
	\pa_t(\ue_2 + \ue_3) - \Delta (d_2^{\eps}\ue_2 + d_3^{\eps}\ue_3) = 0
\end{equation}
to transfer the gradient estimates from $\ue_3$ to $\ue_2$.

\begin{lemma}\label{lem2}
	Assume \eqref{AssumpOnDj} and $d_2 = d_3$. Then it holds
	\begin{equation*}
		\sup_{0<\epsilon< \epsilon_0}\|\na \ue_2\|_{\LQ{2}} \le { C(T,\|u_0\|_{W^{2,q}(\Omega)\times L^{q}(\Omega)^2 \times W^{2,q}(\Omega)}) }
	\end{equation*}
	for {where $q>\max\{N+2;4\}$}.
\end{lemma}
\begin{proof}
	By rewriting \eqref{eq5} as
	\begin{equation}\label{eq6}
		\pa_t(\ue_2 + \ue_3) - \de_2\Delta(\ue_2 + \ue_3) = (\de_3 - \de_2)\Delta \ue_3
	\end{equation}
	then multiplying by $\ue_2 + \ue_3$ gives
	\begin{equation*}
	\begin{aligned}
		&\frac 12\frac{d}{dt}\|\ue_2+\ue_3\|_{\LO{2}}^2 + \de_2\int_{\Omega}|\na (\ue_2+\ue_3)|^2\\ 
		& \hspace*{1cm} = (\de_3 - \de_2)\int_{\Omega}\na\ue_3 \cdot \na(\ue_2+\ue_3)\\
		& \hspace*{1cm} \le \frac{\de_2}{2}\int_{\Omega}|\na(\ue_2+\ue_3)|^2 + \frac{(\de_3-\de_2)^2}{2\de_2}\int_{\Omega}|\na\ue_3|^2.
	\end{aligned}
	\end{equation*}
	Integrating this on $(0,T)$ gives
	\begin{equation*}
		\de_2\|\na(\ue_2+\ue_3)\|_{\LQ{2}}^2 \le \|u_{20}+u_{30}\|_{\LO{2}}^2 + \frac{(\de_3-\de_2)^2}{\de_2}\|\na \ue_3\|_{\LQ{2}}^2.
	\end{equation*}
	Thanks to \eqref{AssumpOnDj} and Lemma \ref{lem1}, it follows that
	\begin{equation*}
		\|\na(\ue_2+\ue_3)\|_{\LQ{2}}^2 \le { C(T,\|u_0\|_{W^{2,q}(\Omega)\times L^{q}(\Omega)^2 \times W^{2,q}(\Omega)}) }.
	\end{equation*} 
	The gradient estimate of $\ue_2$ then follows from
	\begin{equation*}
		\|\na \ue_2\|_{\LQ{2}}^2 \le 2\|\na(\ue_2+\ue_3)\|_{\LQ{2}}^2 + 2\|\nabla \ue_3\|_{\LQ{2}}^2.
	\end{equation*}
\end{proof}
\begin{remark}
	Well-known duality methods show that any $\LQ{p}$-estimate for $\ue_3$ can be transferred to $\ue_2$ provided \eqref{eq5} holds and vice versa, see e.g. \cite{pierre2010global}. Lemma \ref{lem2} shows this transfer of regularity also holds for $\LQ{2}$-estimate of gradients. In fact, it is true for $\LQ{q}$-estimate of gradients for any $1<q<\infty$. Indeed, by \eqref{eq6}
	\begin{equation*}
		\pa_t(\ue_2+\ue_3) - \de_2\Delta(\ue_2+\ue_3) = (\de_3-\de_2)\mathrm{div}(\na\ue_3),
	\end{equation*}
	we can apply results similarly to \cite{byun2007optimal} (for Neumann instead of Dirichlet boundary conditions) to obtain
	\begin{equation*}
		\|\ue_2+\ue_3\|_{L^q(0,T;W^{1,q}(\Omega))}\le C\bra{1+\|\na\ue_3\|_{\LQ{q}}}
	\end{equation*}
	which gives bounds on $\|\na \ue_2\|_{\LQ{q}}$ depending on $\|\na \ue_3\|_{\LQ{q}}$.
\end{remark}
We need another result concerning the well-posedness of the limit system \eqref{ReducedSys}.

\begin{proposition}\label{pro1}
	Assume $d_2 = d_3$. Then for any non-negative, bounded initial data $(u_{10},v_0,u_{40})$, there exists a unique bounded weak solution $(u_1,v,u_4)$ to \eqref{ReducedSys}.
\end{proposition}
\begin{proof}
	Since $d_2=d_3$, the equation of $v$ is reduced to $\pa_t v- \Delta v = 0$, which gives the global well-posedness of a non-negative solution $v$ as well as the uniform bound
	\begin{equation*}
		\sup_{t\ge 1}\|v(t)\|_{\LO{\infty}} \le \|v_0\|_{\LO{p}}.
	\end{equation*}
	With this bound, it is straightforward that the nonlinearities in the equations of $u_1$ and $u_4$ are bounded by linear functions, which implies at once the global existence of bounded weak solutions. Moreover, the nonlinearities are quasi-positive, and therefore the non-negativitiy of solutions follows, provided the non-negavitiy of initial data. Finally, the uniqueness of solutions is guaranteed by their boundedness and local Lipschitz continuity of the nonlinearities.
\end{proof}

It is now ready to prove Theorem \ref{thm1}.
\begin{proof}[Proof of Theorem \ref{thm1}]
	By Lemma \ref{IrRegularityLemCase1},  $ u_1^\epsilon,u_4^\epsilon $ are uniformly bounded in $L^{\infty}(0,T;W^{1,\infty}(\Omega))$, and  $\partial_t u_1^\epsilon,\partial_t u_4^\epsilon $ in $L^{q}(Q_T)$ for any $1<q<\infty$ {if $u_0\in W^{2,q_0+}(\Omega) \times L^{q_0+}(\Omega)^2 \times W^{2,q_0+}(\Omega)$}.
	The Aubin-Lions lemma yields  that   $\{u_1^\epsilon\}_{\eps>0}, \{u_4^\epsilon\}_{\eps>0}$ are relatively compact in $L^{\infty}(Q_T)$. Hence, up to subsequences, 
	\begin{equation*}
		u_1^\epsilon \to u_1, \quad u_4^\epsilon \to u_4 \quad \text{ in }\quad  L^{\infty}(Q_T)\quad \text{ as }\quad \epsilon\to 0^+.
	\end{equation*}
	Now, by Lemmas \ref{lem1} and \ref{lem2}, $\{\na v^\eps\}_{\eps} = \{\na \ue_2 + \na \ue_3\}_{\eps}$ is bounded in $\LQ{2}$. It follows from \eqref{eq6} that $\pa_t v^\eps = \de_2 \Delta v^\eps + (\de_3 - \de_2)\Delta \ue_3$ is bounded in $L^2(0,T;(H^1(\Omega))')$. Another application the Aubin-Lions lemma gives $\{v^{\eps}\}_{\eps}$ is relatively compact in $\LQ{2}$. Due to Lemma \ref{IrRegularityLem} (a), $\{v^\eps\}_{\eps}$ is relatively compact in {$\LQ{q_0 +}$}.

	\medskip
	
	We shall show the strong convergence of $u_{2}^\epsilon$ and $\ue_3$. We have
	\begin{equation*}
	\begin{aligned}
			\left\|\frac{(k_1 u_1^\epsilon + l_2u_4^\epsilon)v^\epsilon }{k_1 u_1^\epsilon + l_2u_4^\epsilon+k_{2}+ l_{1}} - u_3^\epsilon \right\|_{\LQ{2}} &= \left\|\frac{(k_1\ue_1+l_2\ue_4)\ue_2 - (k_2+l_1)\ue_3}{k_1 u_1^\epsilon + l_2u_4^\epsilon+k_{2}+ l_{1}}\right\|_{\LQ{2}}\\
			&\le \frac{1}{k_2+l_1}\|(k_1\ue_1+l_2\ue_4)\ue_2 - (k_2+l_1)\ue_3\|_{\LQ{2}} \to 0
	\end{aligned}
	\end{equation*}
	thanks to Lemma \ref{lem1}. On the other hand, due to the convergence of $\ue_1 \to u_1, \ue_4 \to u_4$ in $\LQ{\infty}$, and of $v^{\eps} \to v$ in {$\LQ{q_0^+}$}, it follows that
	\begin{equation*}
		\frac{(k_1 u_1^\epsilon + l_2u_4^\epsilon)v^\epsilon }{k_1 u_1^\epsilon + l_2u_4^\epsilon+k_{2}+ l_{1}} \to \frac{(k_1u_1+l_2u_4)v}{k_1u_1+l_2u_4+k_2+l_1} \quad \text{ in } \quad \LQ{2}.
	\end{equation*}
	Therefore, $\ue_3 \to u_3$ in $\LQ{2}$ and consequently
	\begin{equation*}
		\ue_3 \to u_3 = \frac{(k_1u_1+l_2u_4)v}{k_1u_1+l_2u_4+k_2+l_1}\quad \text{ in } \quad {\LQ{q_0^+}} . 
	\end{equation*}
 The strong convergence of $\ue_2 \to u_2$ in ${\LQ{q_0^+}}$ follows immediately.

	\medskip	
	It remains to show that $(u_1,v,u_4)$ is the unique bounded to \eqref{ReducedSys}, and hence the convergence $(\ue_j)_{j=1,\ldots, 4}$ to $(u_j)_{j=1,\ldots,4}$ holds for $\eps\to 0$ and not only for a subsequence. By rewriting
	\begin{equation*}
	\begin{aligned}
		- k_1u_1^\epsilon \ue_2 + l_1u_3^\epsilon  &=   (k_1u_1^\epsilon +l_1) \bigg(  u_3^\epsilon - \frac{(k_1 u_1^\epsilon + l_2u_4^\epsilon)v^\epsilon }{k_1 u_1^\epsilon + l_2u_4^\epsilon+k_{2}+ l_{1}} \bigg)      -   \dfrac{(k_1k_2u_1^\epsilon-l_1l_2u_4^\epsilon)v^\epsilon}{k_1u_1^\epsilon + l_2u_4^\epsilon +k_{2} + l_{1}}\\
		&\quad\longrightarrow -\frac{(k_1k_2 u_1 - l_1l_2u_4)v}{k_1u_1+l_2u_4 + k_2+l_1} \; \text{ in } \LQ{2}.
	\end{aligned}
	\end{equation*}
	Thus, for any test function $\varphi \in L^2(0,T;H^1(\Omega))$, letting $\eps\to 0$ in 
	\begin{equation*}
		\iint_{Q_T} \varphi  \pa_t \ue_1   + \de_1\iint_{Q_T} \na \ue_1\cdot \na\varphi = \iint_{Q_T}(-k_1\ue_1\ue_2 + l_1\ue_3)\varphi,
	\end{equation*}
	leads to 
	\begin{equation*}
		\iint_{Q_T} \varphi  \pa_t u_1 +  d_1\iint_{Q_T}\na u_1 \cdot \na\varphi = -\iint_{Q_T}\frac{(k_1k_2 u_1 - l_1l_2u_4)v}{k_1u_1+l_2u_4 + k_2+l_1} \varphi.
	\end{equation*}
	Similarly,
	\begin{equation*}
		\iint_{Q_T}\varphi\pa_t u_4  +  d_4\iint_{Q_T}\na u_4\cdot \na \varphi  = \iint_{Q_T}\frac{(k_1k_2u_1 - l_1l_2u_4)v}{k_1u_1+l_2u_4+k_2+l_1}\varphi.
	\end{equation*}
	For the equation of $v^\eps$ we let $\eps\to 0$ in 
	\begin{equation*}
		\iint_{Q_T}\varphi\pa_t v^{\eps} +  \de_2\iint_{Q_T}\na v^{\eps}\cdot \na \varphi = - (\de_3-\de_2)\iint_{Q_T}\na \ue_3\cdot \na\varphi,
	\end{equation*}
	to conclude, thanks to $\{\pa_tv^\eps\}_{\eps>0}$ is bounded in $L^2(0,T;(H^1(\Omega))')$, $\{\na v^\eps\}_{\eps>0}, \{\na \ue_3\}_{\eps>0}$ is bounded in $\LQ{2}$, and \eqref{AssumpOnDj} with $d_2 = d_3$,
	\begin{equation*}
		\iint_{Q_T}\varphi\pa_t v + d_2\iint_{Q_T}\na v\cdot \na\varphi = 0.
	\end{equation*}
	Therefore $(u_1,v,u_4)$ is a weak solution to \eqref{ReducedSys} as desired.
\end{proof}

\subsection{Proof of Theorem \ref{thm2}}

\begin{lemma}\label{lem3}
	Assume \eqref{AssumpOnDj} and \eqref{eq7} with $p_0$ in \eqref{p0_assump_thm2}. Then there exists $\epsilon_0>0$ such that 
	\begin{equation*}
		\sup_{0<\eps<\eps_0}\sum_{j\in\{1,4\}}\bra{\|\pa_t\ue_j\|_{\LQ{p_2}}+\|\na\ue_j\|_{\LQ{p_3}}} \le { C(T,\|u_0\|_{W^{2,p_0}(\Omega)\times L^{p_0}(\Omega)^2 \times W^{2,p_0}(\Omega)}) }
	\end{equation*}
	for some $\eps>0$ and 
	\begin{equation}\label{p2p3}
		p_2 = \begin{cases}
			\frac{(N+2)p_0}{2(N+2-p_0)} & \text{if } p_0< \frac{N+2}{2},\\
			< p_0 &\text{if } p_0 = \frac{N+2}{2},\\
			p_0 &\text{if } p_0 > \frac{N+2}{2}
		\end{cases}
		\quad \text{ and } \quad 
		p_3 = \begin{cases}
\frac{(N+2)p_0}{2(N+2)-3p_0} &\text{if } p_0 < \frac{N+2}{2},\\
< \frac{(N+2)p_0}{N+2-p_0} &\text{if } p_0=\frac{N+2}{2},\\
			  \frac{(N+2)p_0}{N+2-p_0} &\text{if } \frac{N+2}{2} < p_0 < N+2,\\
			< \infty &\text{if } p_0 \ge N+2.
		\end{cases}
	\end{equation}
\end{lemma}
\begin{proof}
	Thanks to Lemma \ref{IrRegularityLem} (b),  there exists $\epsilon_0>0$ such that 
	\begin{equation*}
	\sup_{0<\eps<\eps_0} \left(	\|\ue_2\|_{\LQ{p_0}} + \|\ue_3\|_{\LQ{p_0}} \right) \le {C\left(T,\|u_{20}\|_{L^{p_0}(\Omega)},\|u_{30}\|_{L^{p_0}(\Omega)}\right)}.
	\end{equation*}
	From the equation of $\ue_1$, we have $\pa_t \ue_1 - \de_1\Delta \ue_1 \le l_1\ue_3$, and therefore by the heat regularisation in Lemma \ref{lem:heat_regularisation} and comparison principle,
	\begin{equation*}
		\|\ue_1\|_{\LQ{p_1}} \le {C(d_1^\epsilon,T,\|u_{10}\|_{W^{2,p_0}(\Omega)})  \le  C( T,\|u_{10}\|_{W^{2,p_0}(\Omega)})}
	\end{equation*}
	where
	\begin{equation}\label{p1}
		p_1 = \begin{cases}
			\frac{(N+2)p_0}{N+2-2p_0} &\text{ if } p_0 < (N+2)/2,\\
			< \infty \text{ arbitrary } &\text{ if } p_0 = (N+2)/2,\\
			 \infty &\text{ if } p_0 > (N+2)/2.
		\end{cases}
	\end{equation}
	It follows that $\pa_t \ue_1 - \de_1\Delta \ue_1 = -k_1\ue_1\ue_2 + l_1\ue_3 \in \LQ{p_2}$, where 
	\begin{equation}\label{p2}
		p_2 = \begin{cases}
				\frac{p_0p_1}{p_0+p_1} &\text{ if } p_1<\infty,\\
				p_0 &\text{ if } p_1 =  \infty,
		\end{cases}
	\end{equation}
	which implies $p_2$ in \eqref{p2p3}.
	Another application of the heat regularisation in Lemma \ref{lem:heat_regularisation} gives
	\begin{equation}\label{p3}
	\sup_{0<\eps<\eps_0} \left( \|\pa_t \ue_1\|_{\LQ{p_2}} + \|\na \ue_1\|_{\LQ{p_3}}  \right) \le C(T), \; \text{ with } \; p_3 = \begin{cases}
			\frac{(N+2)p_2}{N+2-p_2} & \text{ if } p_2<N+2,\\
			< \infty  & \text{ if } p_2 = N+2,\\
			 \infty & \text{ if } p_2 > N+2.
		\end{cases}
	\end{equation}
	Similarly, we have
	\begin{equation*}
		\sup_{0<\eps<\eps_0} \left( \|\pa_t \ue_4\|_{\LQ{p_2}} + \|\na \ue_4\|_{\LQ{p_3}} \right) \le {C(T,\|u_{40}\|_{W^{2,p_0}(\Omega)})}.
	\end{equation*}
Finally, the expression of $p_3$ in  \eqref{p2p3} is obtained by straight computations.  
\end{proof}
\begin{lemma}\label{lem4}
	Assume \eqref{AssumpOnDj}, $d_2 \ne d_3$, and \eqref{eq7} with $p_0$ satisfying \eqref{p0_assump_thm2}. Then 
	\begin{equation}\label{eq7_1}
	\begin{aligned}
	&\sup_{0<\eps<\eps_0}\bra{ \|\na \ue_3\|_{\LQ{2}}^2 +  \frac{1}{\eps}\iint_{Q_T}\left|(k_1\ue_1+l_2\ue_4)\ue_2 - (k_2+l_1)\ue_3\right|^2 } \\
	& \hspace*{2cm} \le { C(T,\|u_0\|_{W^{2,p_0}(\Omega)\times L^{p_0}(\Omega)^2 \times W^{2,p_0}(\Omega)}) }.
	\end{aligned}
	\end{equation}	
\end{lemma}
\begin{proof}
	We use Lemma \ref{TendToManifold} (a) with $p=2$. Due to \eqref{eq7} and Lemma \ref{lem3}, 
	\begin{equation*}
		\|\ue_2\|_{\LQ{p_0}} + \|\pa_t A_2\|_{\LQ{p_2}} + \|\na A_2\|_{\LQ{p_3}} \le { C(T,\|u_0\|_{W^{2,p_0}(\Omega)\times L^{p_0}(\Omega)^2 \times W^{2,p_0}(\Omega)}) }
	\end{equation*}
	for $p_2$ and $p_3$ are in \eqref{p2p3}. We show that
	\begin{equation}\label{eq8}
		\frac{2}{p_0} + \frac{1}{p_2} \le 1.
	\end{equation}
	Note that, by \eqref{p0_assump_thm2},  $p_0>4$ for   $N\ge 3$. If $ p_0 < \frac{N+2}{2}$ with some $N>6$, 
we have  $p_2 = \frac{(N+2)p_0}{2(N+2-p_0)}$ as \eqref{p2p3} and therefore \eqref{eq8} is equivalent to $p_0
\ge \frac{4(N+2)}{N+4}$, which  holds since $4>\frac{4(N+2)}{N+4}$.  In the case  $p_0\ge \frac{N+2}{2} $, we can take  $p_2=p_0^-$ and   (\ref{eq8}) is obvious   since $p_0>4$.

	Due to \eqref{eq8}, we can use H\"older inequality to obtain
	\begin{equation}\label{eq9}
		\iint_{Q_T}(\ue_2)^2\pa_t A_2 \le C(T)\|\ue_2\|_{\LQ{p_0}}^2\|\pa_t A_2\|_{\LQ{p_2}} \le C(T).
	\end{equation}
	For the second term on the right hand side of \eqref{EntropyInequalitya}, we write for any $\sigma \in [0,1)$
	\begin{equation}\label{eq9_1}
	\begin{aligned}
		\iint_{Q_T}(\ue_2)^2\frac{|\na A_2|^2}{A_2} &= \iint_{Q_T}(\ue_2)^2\frac{|\na A_2|^{\frac{2}{1+\sigma}}}{A_2}|\na A_2|^{\frac{2\sigma}{1+\sigma}}\\
		&\le \bra{\iint_{Q_T}|\ue_2|^{p_0}}^{\frac{2}{p_0}}\bra{\iint_{Q_T}\frac{|\na A_2|^2}{A_2^{1+\sigma}}}^{\frac{1}{1+\sigma}}\bra{\iint_{Q_T}|\na A_2|^{\frac{2\sigma}{1+\sigma}\cdot \beta}}^{\frac{1}{\beta}}
	\end{aligned}
	\end{equation}
	where $\beta > 1$ satisfies
	\begin{equation}\label{eq10}
		\frac{2}{p_0} + \frac{1}{1+\sigma}+ \frac{1}{\beta} = 1.
	\end{equation}
Since  $p_0 > 4$,  such a $\beta > 1$ exists  for $\sigma$ close to $1$. It remains to check that, with $\beta$ being computed from \eqref{eq10},
	\begin{equation}\label{eq11}
		\frac{2\sigma}{1+\sigma}\cdot \beta = \frac{2\sigma p_0}{p_0\sigma - 2(1+\sigma)} \le p_3
	\end{equation}
	for some $\sigma \in [0,1)$, with $p_3$ are in \eqref{p2p3}. 
\begin{itemize}
\item If $p_0<\frac{N+2}{2}$ with some $N>6$, then by (\ref{p2p3}) $p_3=\frac{(N+2)p_0}{2(N+2)-3p_0}$ and (\ref{eq11}) becomes 
\begin{align}
p_0\ge \frac{2(3\sigma +1)(N+2)}{\sigma(N+8)} . \label{eq11bbbb}
\end{align}
  In view of (\ref{p0_assump_thm2}), 
$ p_0 >  \frac{8(N+2)}{N+8} $ for all $N\ge 7$. Therefore, we can find a constant $\sigma \in [0,1)$ sufficiently close  to $1$ such that $p_0 > \frac{2(3\sigma+1)(N+2)}{\sigma(N+8)}$, which     implies  (\ref{eq11bbbb}).  
\item If $\frac{N+2}{2} \le  p_0 < N+2$ with some $N\ge 3$, then by (\ref{p2p3}) we can take $p_3=\big(\frac{(N+2)p_0}{ N+2-p_0 }\big)^-$. The condition  (\ref{eq11}) can be followed from  
\begin{align}
p_0 >  \frac{2(2\sigma +1)(N+2)}{\sigma(N+4)} . \label{eq11cccc}
\end{align}
 In view of (\ref{p0_assump_thm2}), $p_0>\frac{6(N+2)}{N+4}$ and therefore there exists $\sigma \in [0,1)$ close  to $1$ such that $p_0 > \frac{2(2\sigma+1)(N+2)}{\sigma(N+8)}$, which is exactly equivalent to (\ref{eq11cccc}).   
\item The case $p_0\ge N+2$ is certain.
\end{itemize}	
	
	 Thanks to \eqref{eq11}, it follows from \eqref{eq9_1} that
	\begin{equation*}
		\iint_{Q_T}(\ue_2)^2\frac{|\na A_2|^2}{A_2} \le C(T)
	\end{equation*}	
	which, in combination with \eqref{eq9}, when inserting into Lemma \ref{TendToManifold} (a) with $p=2$ gives the desired estimate \eqref{eq7_1}. {Here we note that the embedding $W^{2,p_0}(\Omega) \hookrightarrow L^\infty(\Omega)$ holds since $2p_0>N$ for all $N\ge 1$, which  allows to apply Lemma \ref{TendToManifold} with $u_0\in W^{2,p_0}(\Omega)\times L^{p_0}(\Omega)^2 \times W^{2,p_0}(\Omega)$.} 
\end{proof}
\begin{lemma}
	Assume \eqref{AssumpOnDj}, $d_2 \ne d_3$, and \eqref{eq7} with $p_0$ satisfying \eqref{p0_assump_thm2}. Then  
	\begin{equation*}
		\sup_{0<\eps<\eps_0}\|\na\ue_2\|_{\LQ{2}} \le { C(T,\|u_0\|_{W^{2,p_0}(\Omega)\times L^{p_0}(\Omega)^2 \times W^{2,p_0}(\Omega)}) }.
	\end{equation*}
\end{lemma}
\begin{proof}
	The proof is the same as lemma \ref{lem2}.
\end{proof}

We are ready now to prove Theorem \ref{thm2}.
\begin{proof}[Proof of Theorem \ref{thm2}]
	From
	\begin{equation*}
		\pa_t \ue_1 - d_1\Delta \ue_1 = -k_1\ue_1\ue_2 + l_1\ue_3
	\end{equation*}
	where the right hand side is bounded in $\LQ{p_2}$ uniformly in $\eps>0$ with $p_2$ is in \eqref{p2p3} and {$u_0 \in W^{2,q_0+}(\Omega)\times L^{q_0+}(\Omega)^2 \times W^{2,q_0+}(\Omega)$}, it follows that $\{\ue_1\}_{\eps>0}$ is relatively compact in $\LQ{p_2}$. Therefore, up to a subsequence,
	\begin{equation*}
		\ue_1 \to u_1 \quad\text{a.e. in}\quad Q_T.
	\end{equation*}
	Thanks to the fact that $\{\ue_1\}$ is bounded in $\LQ{p_1}$ we get the strong convergence, again up to a subsequence,
	\begin{equation*}
		\ue_1 \to u_1 \quad \text{in}\quad \LQ{p}
	\end{equation*}
	for any $1\le p < p_1$. The same argument gives $\ue_4 \to u_4$ in $\LQ{p}$ for all $1\le p <p_1$. Due to the boundedness of $\{\na \ue_2\}$ and $\{\na \ue_3\}$ in $\LQ{2}$, we obtain from $\pa_tv^\eps = \de_2\Delta v^\eps + (\de_3-\de_2)\Delta \ue_3$ that $\{v^\eps\}_{\eps}$ is relatively compact in $\LQ{2}$. This implies that, up to a subsequence,
	\begin{equation*}
			\frac{(k_1 u_1^\epsilon + l_2u_4^\epsilon)v^\epsilon }{k_1 u_1^\epsilon + l_2u_4^\epsilon+k_{2}+ l_{1}} \to \frac{(k_1u_1+l_2u_4)v}{k_1u_1+l_2u_4+k_2+l_1} \quad \text{ in } \quad \LQ{2}.
	\end{equation*}
	thanks to 
	\begin{equation*}
		\left|\frac{k_1 u_1^\epsilon + l_2u_4^\epsilon}{k_1 u_1^\epsilon + l_2u_4^\epsilon+k_{2}+ l_{1}}\right| \le 1.
	\end{equation*}
	Therefore,
	\begin{equation*}
		\ue_3 \to u_3 = \frac{(k_1u_1+l_2u_4)v}{k_1u_1 + l_2u_4 + k_2 + l_1} \quad \text{in}\quad \LQ{2}.
	\end{equation*}
	Thanks to the above convergence of $\ue_j$, $j=1,\ldots, 4$ and $v^\eps$, it is easy to see that
	\begin{equation*}
	\begin{aligned}
	- k_1u_1^\epsilon \ue_2 + l_1u_3^\epsilon  =&   (k_1u_1^\epsilon +l_1) \bigg(  u_3^\epsilon - \frac{(k_1 u_1^\epsilon + l_2u_4^\epsilon)v^\epsilon }{k_1 u_1^\epsilon + l_2u_4^\epsilon+k_{2}+ l_{1}} \bigg)      -   \dfrac{(k_1k_2u_1^\epsilon-l_1l_2u_4^\epsilon)v^\epsilon}{k_1u_1^\epsilon + l_2u_4^\epsilon +k_{2} + l_{1}}\\
	& \quad \longrightarrow -\frac{(k_1k_2 u_1 - l_1l_2u_4)v}{k_1u_1+l_2u_4 + k_2+l_1} \; \text{ in } \LQ{2}.
	\end{aligned}	
	\end{equation*}
	Therefore, passing to the limit in the weak formulation
	\begin{equation*}
		\iint_{Q_T}\varphi \pa_t \ue_1 + \de_1\iint_{Q_T}\na\ue_1\cdot \na \varphi = \iint_{Q_T}(-k_1\ue_1\ue_2 + l_1\ue_3)\varphi
	\end{equation*}
	we get the weak formulation for $u_1$ in Definition \ref{very_weak_sol} (a). The equations for $\ue_4$ and $v$ follow  the same way, with a remark that for the diffusion term of $v$, it is sufficient to integrate by parts only once since it holds that $v\in L^2(0,T;H^1(\Omega))$. The proof of Theorem \ref{thm2} is finished.
\end{proof}

\subsection{Proof of Theorem \ref{thm3}}
\begin{lemma}
	\label{IrRegularityLemCase2N3} Assume \eqref{AssumpOnDj}. 
	 
$\bullet$ Let $N=1,2$. Then there exists $\epsilon_0>0$ such that    
		\begin{gather}\label{eq12}
			\sup_{0<\eps<\eps_0} \sum_{j=1,4}\bra{\|\pa_t\ue_j\|_{\LQ{2+}}+\|\Delta \ue_j\|_{\LQ{2+}} + \|\ue_j\|_{\LQ{\infty}} + \|\na\ue_j\|_{\LQ{r_0+}}} \le C(T),
		\end{gather} 
		{where $r_0 = 2(N+2)/N$ and  the constant depends on $\|u_0\|_{W^{2,2+}(\Omega)\times L^{2+}(\Omega)^2 \times W^{2,2+}(\Omega)}$}.\\

$\bullet$	Let $N\ge 3$. If  $d_2,d_3$  satisfy \eqref{e0} with $p_0$ in \eqref{p0_assump_thm3} then there exists $\epsilon_0>0$ such that      
		\begin{align}
			\sup_{0<\eps<\eps_0} \sum_{j=1,4}\bra{\|\pa_t\ue_j\|_{\LQ{s_2+}} + \|\Delta\ue_j\|_{\LQ{s_2+}}} &\le C(T), \quad s_2 = \frac{3(N+2)}{2N+2},\label{RegularityU1VIrCase2N3}\\
		\sup_{0<\eps<\eps_0}	\sum_{j=1,4}\bra{\|\ue_j\|_{\LQ{s_3+}} + \|\na \ue_j\|_{\LQ{s_4+}}} &\le C(T), \quad s_3 = \frac{3(N+2)}{N-2}, s_4 = \frac{3(N+2)}{N+1},
			\label{RegularityU1VIrCase2N3b}
		\end{align} 
		{where the constant depends on $\|u_0\|_{W^{2,p_0}(\Omega)\times L^{p_0}(\Omega)^2 \times W^{2,p_0}(\Omega)}$}.

\end{lemma} 
\begin{proof} 
	For $N=1,2$, it follows from Lemma \ref{IrRegularityLem} that $\pa_t\ue_1 - \de_1\Delta \ue_1 \le l_1\ue_3 \in \LQ{2+}$, and   $$ \sup_{0<\eps<\eps_0} \|\ue_1\|_{\LQ{\infty}} \le {C(T,\|u_0\|_{W^{2,2+}(\Omega)\times L^{2+}(\Omega)^2 \times W^{2,2+}(\Omega)})}$$ due to Lemma \ref{lem:heat_regularisation}. Consequently, $\pa_t\ue_1-\de_1\Delta\ue_1 = -k_1\ue_1\ue_2 + l_1\ue_3 \in \LQ{2+}$, which leads to the desired estimate of $\pa_t\ue_1$, $\Delta \ue_1$, and $\na \ue_1$ thanks to another application of Lemma \ref{lem:heat_regularisation}. The estimates for $\ue_4$ follow the same way.
	
	\medskip
	The remaining estimates can be shown similarly as the proof of Lemma \ref{lem:Estimates_u1u4} by making uses of the heat regularisation (see Lemm \ref{lem:heat_regularisation})
	and the comparison principle. 
\end{proof}

\begin{lemma} \label{lem5} Assume \eqref{AssumpOnDj} with $d_2 \ne d_3$. Let $N=1,2$, or $N\ge 3$ and assuming \eqref{e0} with $p_0$ in  \eqref{p0_assump_thm3}. Then  
	\begin{align}\label{eq13}
	\sup_{0<\epsilon<\epsilon_0} \bigg( \frac{1}{\epsilon^{\frac{1}{2(3-\delta)}}} \bigg\| u_3^\epsilon - \frac{ (k_1 u_1^\epsilon + l_2u_4^\epsilon) v^\epsilon   }{k_1u_1^\epsilon + l_2u_4^\epsilon +k_{2}+ l_{1}}  \bigg\|_{\LQ{r}}\bigg) \le C(T)
	\end{align} 
	for some $\delta>0$ small enough, where $r = 4/3$ for $N=1,2$ and $r=6/5$ for $N\ge 3$.
\end{lemma}
\begin{proof}
	\underline{When $N=1,2$}, we write $p = 1+\delta$ for some $0< \delta<1$. Then we have the following estimates
	\begin{equation*}
		\left|\iint_{Q_T}(\ue_2)^{1+\delta}(A_2+1)^{\delta}\right| \le \|A_2+1\|_{\LQ{\infty}}\iint_{Q_T}|\ue_2|^{1+\delta} \le C(T),
	\end{equation*}	
	\begin{equation*}
		\left|\iint_{Q_T}(\ue_2)^{1+\delta}|\pa_tA_2|\right| \le \bra{\iint_{Q_T}|\ue_2|^{2+\delta}}^{\frac{1+\delta}{2+\delta}}\bra{\iint_{Q_T}|\pa_tA_2|^{2+\delta}}^{\frac{1}{2+\delta}} \le C(T)
	\end{equation*}
	for $\delta > 0$ small enough thanks to Lemmas \ref{IrRegularityLem} and \ref{IrRegularityLemCase2N3}.   Finally, 
	\begin{align*}
		\left|\iint_{Q_T}(\ue_2)^{1+\delta}|\na A_2|^2\right| \le \bra{\iint_{Q_T}|\ue_2|^{2+\delta}}^{\frac{1+\delta}{2+\delta}}\bra{\iint_{Q_T}|\na A_2|^{2(2+\delta)}}^{\frac{1}{2+\delta}} \le C(T)
	\end{align*}	
	for $\delta>0$ small enough, where we used \eqref{eq12} taking into account $r_0 = 6$ for $N=1$ and $r_0 = 4$ for $N=2$. Note that {  the above constants depend  on $\|u_0\|_{W^{2,2+}(\Omega)\times L^{2+}(\Omega)^2 \times W^{2,2+}(\Omega)}$}. We can now apply Lemma \ref{TendToManifold} (b) with $p=1+\delta$ to get in particular
	\begin{equation*}
		\frac{1}{\eps^{\frac{1}{3-\delta}}}\iint_{Q_T}\frac{\big|(A_2+\eps^{\frac{1}{3-\delta}})\ue_2 - A_3\ue_3\big|^2}{\bra{(A_2+\eps^{\frac{1}{3-\delta}})\ue_2 + A_3\ue_3}^{1-\delta}} \le { C(T,\|u_0\|_{L^{\infty}(\Omega) \times L^{1+\delta}(\Omega)^2 \times L^{\infty}(\Omega)}) }.
	\end{equation*}

	By virtue of Lemmas \ref{IrRegularityLem} and  \ref{IrRegularityLemCase2N3}, we can apply the H\"older inequality to see that  
	\begin{equation*}
	\begin{aligned}
	& \iint_{Q_T} \bigg|   u_3^\epsilon - \frac{(A_2 + \epsilon^{\frac{1}{3-\delta}})v^\eps}{A_2 + \epsilon^{\frac{1}{3-\delta}} +A_3}    \bigg|^{\frac{4}{3-\delta}}  \le  \frac{1}{A_3^{\frac{4}{3-\delta}}}\iint_{Q_T} \big | (A_2+\epsilon^{\frac{1}{3-\delta}}) \ue_2   -   A_3  u_3^\epsilon    \big |^{\frac{4}{3-\delta}}\\
	& \le  C  \bigg( \iint_{Q_T}   \frac{\big|  (A_2 +\epsilon^{\frac{1}{3-\delta}}) \ue_2 -  A_3  u_3^\epsilon  \big|^2}{\big|(A_2 +\epsilon^{\frac{1}{4-p_0}}) u_2^\epsilon  +   A_3  u_3^\epsilon\big|^{1-\delta}} \bigg)^{\frac{2}{3-\delta}}   \times \bigg(\iint_{Q_T} \big|(A_2 +\epsilon^{\frac{1}{3-\delta}}) \ue_2  +  A_3 u_3^\epsilon\big|^{2} \bigg)^{\frac{1-\delta}{3-\delta}}\\
	& \le   C \epsilon^{\frac{2}{(3-\delta)^2}} \Big( \|A_2 +1\|^2 _{L^{\infty}(Q_T)} \|u_2^\epsilon  \|_{L^{2}(Q_T)}^2  +   A_3^2 \| u_3^\epsilon \|_{L^{2}(Q_T)}^2  \Big)^{\frac{1-\delta}{3-\delta}}\\
	&\le {C(T,\|u_0\|_{W^{2,2+}(\Omega)\times L^{2+}(\Omega)^2 \times W^{2,2+}(\Omega)})} \, \eps^{\frac{2}{(3-\delta)^2}}
	\end{aligned}
	\end{equation*}
	which yields
	\begin{equation*}
		\left\| u_3^\epsilon - \frac{(A_2 + \epsilon^{\frac{1}{3-\delta}})v^\eps}{A_2 + \epsilon^{\frac{1}{3-\delta}} +A_3}\right\|_{\LQ{\frac{4}{3-\delta}}} \le C(T)\eps^{\frac{1}{2(3-\delta)}}.
	\end{equation*}

	On the other hand, one can  observe 
	\begin{align}\label{e3}
	\bigg|   \frac{ (k_1 u_1^\epsilon + l_2u_4^\epsilon)v^\epsilon }{k_1u_1^\epsilon + l_2u_4^\epsilon +  k_{2}+ l_{1}}   - \frac{(A_2 + \epsilon^{\frac{1}{3-\delta}}) v^\epsilon }{A_2 + \epsilon^{\frac{1}{3-\delta}} +A_3}  \bigg|  \le    \frac{1}{k_{2}+ l_{1}} v^\epsilon\, \epsilon^{\frac{1}{3-\delta}}.  
	\end{align}
	Therefore, by the triangle inequality,  
	\begin{align}
	&\bigg \|  u_3^\epsilon - \frac{ (k_1 u_1^\epsilon + l_2u_4^\epsilon)v^\epsilon }{k_1u_1^\epsilon +l_2u_4^\epsilon   +k_{2}+ l_{1}}    \bigg\|_{L^{\frac{4}{3-\delta}}(Q_T)}\\
	&\le    \bigg \|  u_3^\epsilon - \frac{(A_2 + \epsilon^{\frac{1}{4-p_0}})v^\epsilon }{A_2 + \epsilon^{\frac{1}{3-\delta}} +A_3}    \bigg\|_{L^{\frac{4}{3-\delta}}(Q_T)} +  \frac{\epsilon^{\frac{1}{3-\delta}}}{k_{2}+ l_{1}} \| v^\epsilon \|_{L^{\frac{4}{3-\delta}}(Q_T)} 
	\le C(T)\eps^{\frac{1}{2(3-\delta)}} \label{SlowProofCase2b}
	\end{align}
	which finishes the proof of \eqref{eq13} in the case $N=1,2$.
	
	\medskip
	\underline{For $N\ge 3$},  thanks to \eqref{p0_assump_thm3} and  \eqref{e0}  we have
	\begin{equation}\label{e2}
		\|\ue_2\|_{\LQ{p_0}} + \|\ue_3\|_{\LQ{p_0}} \le  C(T,\|u_{20}\|_{L^{p_0}(\Omega)},\|u_{30}\|_{L^{p_0}(\Omega)}).
	\end{equation}
{By choosing $\delta>0$ small enough such that} 
\begin{align} \label{p0inqN3}
{p_0>   \frac{3\delta(N+2) + 3(N+2)}{N+4} \quad \text{or equivalently} \quad \frac{p_0}{p_0-1-\delta}\le \frac{3(N+2)}{2N+2},}  
\end{align}	
we can estimate using H\"older inequality	
	\begin{align}\label{eq14}
		\left|\iint_{Q_T}(\ue_2)^{1+\delta}(A_2+1)^{\delta} \right| \le& \bra{\iint_{Q_T}(\ue_2)^{p_0}}^{\frac{1+\delta}{p_0}}\bra{\iint_{Q_T}(A_2+1)^{\frac{\delta p_0}{p_0-1-\delta}}}^{\frac{p_0-1-\delta}{p_0}} \le C(T),
	\\
		\label{eq15} \left|\iint_{Q_T}(\ue_2)^{1+\delta}\pa_tA_2\right| \le& \bra{\iint_{Q_T}(\ue_2)^{p_0}}^{\frac{1+\delta}{p_0}}\bra{\iint_{Q_T}|\pa_tA_2|^{\frac{p_0}{p_0-1-\delta}}}^{\frac{p_0-1-\delta}{p_0}} \le C(T).
	\end{align}
	 Similarly,
	\begin{align} \label{eq16}
		\left|\iint_{Q_T}(\ue_2)^{1+\delta}|\na A_2|^2\right| \le\bra{\iint_{Q_T}(\ue_2)^{p_0}}^{\frac{1+\delta}{p_0}}\bra{\iint_{Q_T}|\na A_2|^{\frac{2p_0}{p_0-1-\delta}}}^{\frac{p_0-1-\delta}{p_0}} \le C(T) 
	\end{align}
	by using  \eqref{RegularityU1VIrCase2N3b}, \eqref{p0inqN3}.  {We note       $C(T)=C(T, \|u_0\|_{W^{2,p_0}(\Omega)\times L^{p_0}(\Omega)^2 \times W^{2,p_0}(\Omega)})$ in the above estimates.}
	
	\medskip   
	
Therefore, inserting \eqref{eq14}--\eqref{eq16} into Lemma \ref{TendToManifold} (b) with $p=1+\delta$, $\delta>0$ small enough, leads to 
	\begin{equation*}
		\frac{1}{\eps^{\frac{1}{3-\delta}}}\iint_{Q_T}\frac{\big|(A_2+\eps^{\frac{1}{3-\delta}})\ue_2 - A_3\ue_3\big|^2}{\bra{(A_2+\eps^{\frac{1}{3-\delta}})\ue_2 + A_3\ue_3}^{1-\delta}} \le {C(T, \|u_0\|_{W^{2,q}(\Omega)\times L^{q}(\Omega)^2 \times W^{2,q}(\Omega)})}
	\end{equation*}	
	where $q = \max\{N,p_0,(N+2)/2\}$. By applying the H\"older's inequality,
	\begin{equation*} 
	\begin{aligned}
		 \iint_{Q_T} \bigg|   u_3^\epsilon - \frac{(A_2 + \epsilon^{\frac{1}{3-\delta}})v^\epsilon}{A_2 + \epsilon^{\frac{1}{3-\delta}} +A_3}    \bigg|^{\frac{6(N+2)}{5N+8}}  
		&\le C    
		\iint_{Q_T} \big|    (A_2+\epsilon^{\frac{1}{3-\delta}}) \ue_2   -   A_3  u_3^\epsilon \big|^{\frac{6(N+2)}{5N+8}}\\
		&\le    
		C 
		\bigg( \iint_{Q_T}   \frac{\big|  (A_2 +\epsilon^{\frac{1}{3-\delta}}) \ue_2  -   A_3  u_3^\epsilon  \big|^2}{\big|(A_2 +\epsilon^{\frac{1}{3-\delta}})\ue_2  +   A_3  u_3^\epsilon\big|^{1-\delta}} \bigg)^{\frac{3(N+2)}{5N+8}}\\
		&\qquad \times \bigg(\iint_{Q_T} \big|(A_2 +\epsilon^{\frac{1}{3-\delta}}) \ue_2  +   A_3 u_3^\epsilon\big|^{(1-\delta)\frac{3(N+2)}{2N+2}} \bigg)^{\frac{ 2N+2 }{5N+8}}\\
		&\le C(T)\bra{\eps^{\frac{1}{3-\delta}}}^{\frac{3(N+2)}{5N+8}}
	\end{aligned}
	\end{equation*}
	thanks to \eqref{e2}, \eqref{RegularityU1VIrCase2N3b}, and
	\begin{equation} \label{SlowCaseN3bbb} 
	\begin{aligned}
	&	\left\|(A_2 +\epsilon^{\frac{1}{4-p_0}}) \ue_2  +   A_3  u_3^\epsilon \right\|_{L^{\frac{3(N+2)}{2N+2}}(Q_T)} \\
	& \hspace*{2cm}  \le \|A_2 +1\|_{\LQ{\frac{3(N+2)}{N-2}}}  \|u_2^\epsilon  \|_{\LQ{\frac{3(N+2)}{N+4}}}  +  A_3\|u_3^\epsilon \|_{L^{\frac{3(N+2)}{2N+2}}(Q_T)}.  
	\end{aligned}
\end{equation}	 
	Thus
	\begin{equation*}
		\bigg \|  u_3^\epsilon - \frac{( A_2 + \epsilon^{\frac{1}{4-p_0}}) v^\epsilon }{A_2 + \epsilon^{\frac{1}{3-\delta}} +A_3}    \bigg\|_{L^{\frac{6(N+2)}{5N+8}}(Q_T)} \le C(T)\eps^{\frac{1}{2(3-\delta)}}.
	\end{equation*}
	Using \eqref{e3} again we obtain finally
	\begin{align*}
	&\bigg \|  u_3^\epsilon - \frac{(k_1 u_1^\epsilon + l_2u_4^\epsilon)v^\epsilon  }{k_1u_1^\epsilon +l_2u_4^\epsilon   +k_{2}+ l_{1}}  \bigg\|_{L^{\frac{6(N+2)}{5N+8}}(Q_T)}\\
	&\le    \bigg \|  u_3^\epsilon - \frac{(A_2 + \epsilon^{\frac{1}{4-p_0}})v^\epsilon }{A_2 + \epsilon^{\frac{1}{3-\delta}} +A_3}    \bigg\|_{L^{\frac{6(N+2)}{5N+8}}(Q_T)} +  \frac{\epsilon^{\frac{1}{3-\delta}}}{k_{2}+ l_{1}} \| v^\epsilon \|_{L^{\frac{6(N+2)}{5N+8}}(Q_T)} 
	\le C(T)\eps^{\frac{1}{2(3-\delta)}} 
	\end{align*}
	which proves \eqref{eq13} in the case $N\ge 3$.
\end{proof}

\begin{proof}[Proof of Theorem \ref{thm3}]
	The convergence of  the critical manifold was shown in Lemma \ref{lem5}. From the the equation $\pa_t \ue_1 - \de_1\Delta \ue_1 = -k_1\ue_1\ue_2 + l_1\ue_3$ in which the right hand side is bounded in $\LQ{1}$ we have, up to a subsequence, $\ue_1 \to u_1$ in $\LQ{1}$. Combining this with the boundedness of $\{\ue_1\}_{\eps>0}$ (as Lemma \ref{IrRegularityLemCase2N3}) yields
	\begin{equation*}
		{ \ue_1 \to u_1 \quad \text{ in } \quad \left\{ \begin{array}{llll}
 L^p(Q_T), 1\le p<\infty & \text{if } N=1,2, \vspace*{0.15cm}\\
\displaystyle L^{\frac{3(N+2)}{N-2}}(Q_T) & \text{if } N\ge 3.	
\end{array}		 \right. }  
	\end{equation*}
	Similarly,
	\begin{equation}\label{eq17}
		{ \ue_4 \to u_4 \quad \text{ in } \quad \left\{ \begin{array}{llll}
 L^p(Q_T), 1\le p<\infty & \text{if } N=1,2, \vspace*{0.15cm}\\
\displaystyle L^{\frac{3(N+2)}{N-2}}(Q_T) & \text{if } N\ge 3.	
\end{array}		 \right.}  
	\end{equation} 
	 From \eqref{L2plus}, it holds
	\begin{equation}\label{eq18}
		\ue_j \rightharpoonup u_j \quad\text{weakly in }\quad \LQ{2}, \,j=2,3.
	\end{equation}
	By using \eqref{L2plus}, \eqref{eq13}, and the fact that $\dfrac{k_1k_2u_1^\epsilon-l_1l_2u_4^\epsilon}{k_1u_1^\epsilon + l_2u_4^\epsilon +k_{2} + l_{1}}$ is bounded uniformly in $\LQ{\infty}$, we obtain
	\begin{equation}\label{eq19}
	\begin{aligned}
	- k_1u_1^\epsilon \ue_2 + l_1u_3^\epsilon  =   (k_1u_1^\epsilon +l_1) \bigg(  u_3^\epsilon - \frac{(k_1 u_1^\epsilon + l_2u_4^\epsilon)v^\epsilon }{k_1 u_1^\epsilon + l_2u_4^\epsilon+k_{2}+ l_{1}} \bigg)      -   \dfrac{k_1k_2u_1^\epsilon-l_1l_2u_4^\epsilon}{k_1u_1^\epsilon + l_2u_4^\epsilon +k_{2} + l_{1}}v^\epsilon\\
	\longrightharpoonup -\frac{(k_1k_2 u_1 - l_1l_2u_4)v}{k_1u_1+l_2u_4 + k_2+l_1} \; \text{ weakly in } \LQ{2}.
	\end{aligned}	
	\end{equation}
	From \eqref{eq17}, \eqref{eq18} and \eqref{eq19}, it is enough to pass to the limit in (very) weak formulation of the system involving $(\ue_1,v^\eps,\ue_4)$ to conclude that $(u_1,v,u_4)$ is a very weak solution to the reduced system \eqref{ReducedSys} according to Definition \ref{very_weak_sol} (b). This finishes the proof of Theorem \ref{thm3}.
\end{proof}

\subsection{Proof of Theorem \ref{thm4}}
 \begin{proof}[Proof of Theorem \ref{thm4}]
	Thanks to \eqref{L2plus}, we can use the equation of $\ue_1$ and $\ue_4$ to obtain that
	\begin{equation*}
		\ue_j \to u_j \; \text{ in } \; \LQ{2}, \quad \na \ue_j \rightharpoonup \na u_j \; \text{weakly in}\; \LQ{2}, \quad j=1,4. 
	\end{equation*} 	
	We have
	\begin{align*}
		\left|\iint_{Q_T}\sbra{(k_1\ue_1 + l_2\ue_4)\ue_2 - (k_2 + l_1)\ue_3}\psi \right| &= \eps\left|\iint_{Q_T}(\pa_t \ue_2 - \de_2\Delta \ue_2)\psi \right|\\
		&= \eps\left|\iint_{Q_T}(\pa_t \psi + \de_2\Delta \psi)\ue_2 \right|\\
		&\le \eps\|\pa_t \psi + \de_2\Delta\psi\|_{\LQ{2}}\|\ue_2\|_{\LQ{2}}\\
		&= O(\eps)
	\end{align*}
	for all $\psi \in  C_c^{\infty}(Q_T)$. Note that $(k_1\ue_1 + l_2\ue_4)\ue_2 - (k_2 + l_1)\ue_3$ is bounded in $\LQ{1+}$ uniformly w.r.t. $\eps>0$, we obtain 
	\begin{equation}\label{f2}
		(k_1\ue_1 + l_2\ue_4)\ue_2 - (k_2 + l_1)\ue_3 \rightarrow 0 \quad \text{ in  distributional sense}.
	\end{equation}
We have  $(k_1\ue_1+l_2\ue_4+k_2+l_1)^{-1} \to (k_1u_1+l_2u_4+k_2+l_1)^{-1}$ in $\LQ{q}$ for all $1\le q<+\infty$ since $(k_1\ue_1+l_2\ue_4+k_2+l_1)^{-1}$ is uniformly bounded  in $\LQ{\infty}$. Therefore
	\begin{equation}\label{f3}
	\begin{aligned}
		\ue_3 - \frac{(k_1\ue_1+l_2\ue_4)v^\eps}{k_1\ue_1+l_2\ue_4+k_2+l_1} &= \frac{1}{k_1\ue_1+l_2\ue_4+k_2+l_1}\sbra{(k_2 + l_1)\ue_3 - (k_1\ue_1 + l_2\ue_4)\ue_2} \to 0   
	\end{aligned}
	\end{equation}
in distributional sense,	which is the convergence of critical manifold in \eqref{f1}.
	{By rewriting
	\begin{equation*}
		- k_1u_1^\epsilon \ue_2 + l_1u_3^\epsilon = -\frac{k_1\ue_1+l_1}{k_1\ue_1+l_2\ue_4 + k_2+l_1}\sbra{(k_1\ue_1+l_2\ue_4)\ue_2 - (k_2+l_1)\ue_3} - \dfrac{(k_1k_2u_1^\epsilon-l_1l_2u_4^\epsilon)v^\epsilon}{k_1u_1^\epsilon + l_2u_4^\epsilon +k_{2} + l_{1}}
	\end{equation*}
	and using \eqref{f2} we get
	\begin{equation*}
		- k_1u_1^\epsilon \ue_2 + l_1u_3^\epsilon 
		\to  -\frac{(k_1k_2 u_1 - l_1l_2u_4)v}{k_1u_1+l_2u_4 + k_2+l_1} \quad \text{in distributional sense},
	\end{equation*}
	which is enough to pass to the limit in the equation of $\ue_1$. The same holds for the equation of $\ue_4$. From \eqref{f3} we also obtain
	\begin{equation*}
		\ue_3 \longrightharpoonup \frac{(k_1u_1 + l_2u_4)v}{k_1u_1 + l_2u_3 + k_2 + l_1} \quad \text{weakly in} \quad \LQ{1+},
	\end{equation*}
	which allows to pass to the limit $\eps \to 0$ in the equation of $v^\eps$ to ultimately conclude that $(u_1,v,u_2)$ is a very weak solution to \eqref{ReducedSys}. The proof of Theorem \ref{thm4} is finished}.
 \end{proof}

\medskip

\noindent{\textbf{Acknowledgement.}} This research is funded by the FWF project ``Quasi-steady-state approximation for PDE", number I-5213.


\begin{thebibliography}{00}
	
	\bibitem[BCDK21]{brocchieri2021evolution}
	Elisabetta Brocchieri, Lucilla Corrias, Helge Dietert, and Yong-Jung Kim.
	\newblock Evolution of dietary diversity and a starvation driven
	cross-diffusion system as its singular limit.
	\newblock {\em Journal of Mathematical Biology}, 83(5):1--40, 2021.
	
	\bibitem[BD06]{bisi2006reactive}
	M~Bisi and L~Desvillettes.
	\newblock From reactive boltzmann equations to reaction--diffusion systems.
	\newblock {\em Journal of statistical physics}, 124:881--912, 2006.
	
	\bibitem[BH03]{bothe2003reaction}
	Dieter Bothe and Danielle Hilhorst.
	\newblock A reaction--diffusion system with fast reversible reaction.
	\newblock {\em Journal of mathematical analysis and applications},
	286(1):125--135, 2003.
	
	\bibitem[BP12]{bothe2012instantaneous}
	Dieter Bothe and Michel Pierre.
	\newblock The instantaneous limit for reaction-diffusion systems with a fast
	irreversible reaction.
	\newblock {\em Discrete \& Continuous Dynamical Systems-S}, 5(1):49, 2012.
	
	\bibitem[BPR12]{bothe2012cross}
	Dieter Bothe, Michel Pierre, and Guillaume Rolland.
	\newblock Cross-diffusion limit for a reaction-diffusion system with fast
	reversible reaction.
	\newblock {\em Communications in Partial Differential Equations},
	37(11):1940--1966, 2012.
	
	\bibitem[Byu07]{byun2007optimal}
	Sun-Sig Byun.
	\newblock Optimal ${W}^{1,p}$ regularity theory for parabolic equations in
	divergence form.
	\newblock {\em Journal of Evolution Equations}, 7(3):415--428, 2007.
	
	\bibitem[CDF14]{canizo2014improved}
	Jos{\'e}~A Canizo, Laurent Desvillettes, and Klemens Fellner.
	\newblock Improved duality estimates and applications to reaction-diffusion
	equations.
	\newblock {\em Communications in Partial Differential Equations},
	39(6):1185--1204, 2014.
	
	\bibitem[CDS18]{conforto2018reaction}
	Fiammetta Conforto, Laurent Desvillettes, and Cinzia Soresina.
	\newblock About reaction--diffusion systems involving the holling-type ii and
	the beddington--deangelis functional responses for predator--prey models.
	\newblock {\em Nonlinear Differential Equations and Applications NoDEA},
	25(3):1--39, 2018.
	
	\bibitem[DDJ20]{daus2020cross}
	Esther~S Daus, Laurent Desvillettes, and Ansgar J{\"u}ngel.
	\newblock Cross-diffusion systems and fast-reaction limits.
	\newblock {\em Bulletin des Sciences Math{\'e}matiques}, 159:102824, 2020.
	
	\bibitem[DS19]{desvillettes2019non}
	Laurent Desvillettes and Cinzia Soresina.
	\newblock Non-triangular cross-diffusion systems with predator--prey reaction
	terms.
	\newblock {\em Ricerche di Matematica}, 68(1):295--314, 2019.
	
	\bibitem[DT15]{desvillettes2015new}
	Laurent Desvillettes and Ariane Trescases.
	\newblock New results for triangular reaction cross diffusion system.
	\newblock {\em Journal of Mathematical Analysis and Applications},
	430(1):32--59, 2015.
	
	\bibitem[EMT20]{einav2020indirect}
	Amit Einav, Jeffrey~J Morgan, and Bao~Q Tang.
	\newblock Indirect diffusion effect in degenerate reaction-diffusion systems.
	\newblock {\em SIAM Journal on Mathematical Analysis}, 52(5):4314--4361, 2020.
	
	\bibitem[Eva80]{evans1980convergence}
	Lawrence~C Evans.
	\newblock A convergence theorem for a chemical diffusion-reaction system.
	\newblock In {\em Houston J. Math}. Citeseer, 1980.
	
	\bibitem[FLWW18]{FraLaxWalWit18}
	Martin Frank, Christian Lax, Sebastian Walcher, and Olaf Wittich.
	\newblock Quasi-steady state reduction for the michaelis--menten
	reaction--diffusion system.
	\newblock {\em Journal of Mathematical Chemistry}, 56(6):1759--1781, 2018.
	
	\bibitem[FMT20]{fellner2020global}
	Klemens Fellner, Jeff Morgan, and Bao~Quoc Tang.
	\newblock Global classical solutions to quadratic systems with mass control in
	arbitrary dimensions.
	\newblock {\em Annales de l'Institut Henri Poincar{\'e} C, Analyse non
		lin{\'e}aire}, 37(2):281--307, 2020.
	
	\bibitem[GWZ15]{goeke2015determining}
	Alexandra Goeke, Sebastian Walcher, and Eva Zerz.
	\newblock Determining “small parameters” for quasi-steady state.
	\newblock {\em Journal of Differential Equations}, 259(3):1149--1180, 2015.
	
	\bibitem[Hen03]{henri1903lois}
	Victor Henri.
	\newblock {\em Lois g{\'e}n{\'e}rales de l'action des diastases}.
	\newblock Librairie Scientifique A. Hermann, 1903.
	
	\bibitem[Hen06]{henri2006theorie}
	Victor Henri.
	\newblock Th{\'e}orie g{\'e}n{\'e}rale de l'action de quelques diastases par
	victor henri [cr acad. sci. paris 135 (1902) 916-919].
	\newblock {\em Comptes rendus. Biologies}, 329(1):47--50, 2006.
	
	\bibitem[HT16]{henneke2016fast}
	Felix Henneke and Bao~Q Tang.
	\newblock Fast reaction limit of a volume--surface reaction--diffusion system
	towards a heat equation with dynamical boundary conditions.
	\newblock {\em Asymptotic Analysis}, 98(4):325--339, 2016.
	
	\bibitem[IMMN17]{iida2017vanishing}
	M~Iida, H~Monobe, H~Murakawa, and H~Ninomiya.
	\newblock Vanishing, moving and immovable interfaces in fast reaction limits.
	\newblock {\em Journal of Differential Equations}, 263(5):2715--2735, 2017.
	
	\bibitem[IMN06]{iida2006diffusion}
	Masato Iida, Masayasu Mimura, and Hirokazu Ninomiya.
	\newblock Diffusion, cross-diffusion and competitive interaction.
	\newblock {\em Journal of mathematical biology}, 53(4):617--641, 2006.
	
	\bibitem[KKK{\etalchar{+}}07]{kalachev2007reduction}
	Leonid~V Kalachev, Hans~G Kaper, Tasso~J Kaper, Nikola Popovic, and Antonios
	Zagaris.
	\newblock Reduction for michaelis-menten-henri kinetics in the presence of
	diffusion.
	\newblock {\em Electronic Journal of Differential Equations (EJDE)[electronic
		only]}, 2007:155--184, 2007.
	
	\bibitem[Kue15]{kuehn2015multiple}
	Christian Kuehn.
	\newblock {\em Multiple time scale dynamics}, volume 191.
	\newblock Springer, 2015.
	
	\bibitem[Lam87]{lamberton1987equations}
	Damien Lamberton.
	\newblock Equations d'{\'e}volution lin{\'e}aires associ{\'e}es {{a}} des
	semi-groupes de contractions dans les espaces ${L}^p$.
	\newblock {\em Journal of functional analysis}, 72(2):252--262, 1987.
	
	\bibitem[LSU88]{ladyvzenskaja1988linear}
	Ol'ga~A Lady{\v{z}}enskaja, Vsevolod~Alekseevich Solonnikov, and Nina~N
	Ural'ceva.
	\newblock {\em Linear and quasi-linear equations of parabolic type}, volume~23.
	\newblock American Mathematical Soc., 1988.
	
	\bibitem[MJ80]{martin1980mathematical}
	Robert~H Martin~Jr.
	\newblock Mathematical models in gas-liquid reactions.
	\newblock {\em Nonlinear Analysis: Theory, Methods \& Applications},
	4(3):509--527, 1980.
	
	\bibitem[MN11]{murakawa2011fast}
	Hideki Murakawa and Hirokazu Ninomiya.
	\newblock Fast reaction limit of a three-component reaction--diffusion system.
	\newblock {\em Journal of mathematical analysis and applications},
	379(1):150--170, 2011.
	
	\bibitem[Per15]{perthame2015parabolic}
	Beno{\^\i}t Perthame.
	\newblock Parabolic equations in biology: growth, reaction, movement and
	diffusion.
	\newblock Lecture Notes on Mathematical Modelling in Life Sciences. Springer,
	2015.
	
	\bibitem[Pie10]{pierre2010global}
	Michel Pierre.
	\newblock Global existence in reaction-diffusion systems with control of mass:
	a survey.
	\newblock {\em Milan Journal of Mathematics}, 78(2):417--455, 2010.
	
	\bibitem[PS22]{perthame2022fast}
	Beno{\^\i}t Perthame and Jakub Skrzeczkowski.
	\newblock Fast reaction limit with nonmonotone reaction function.
	\newblock {\em Communications on Pure and Applied Mathematics}, 2022.
	
	\bibitem[SKT79]{shigesada1979spatial}
	Nanako Shigesada, Kohkichi Kawasaki, and Ei~Teramoto.
	\newblock Spatial segregation of interacting species.
	\newblock {\em Journal of theoretical biology}, 79(1):83--99, 1979.
	
	\bibitem[SS89]{segel1989quasi}
	Lee~A Segel and Marshall Slemrod.
	\newblock The quasi-steady-state assumption: a case study in perturbation.
	\newblock {\em SIAM review}, 31(3):446--477, 1989.
	
\end{thebibliography}

\newcommand{\etalchar}[1]{$^{#1}$}

\end{document}